\documentclass[11pt,a4paper]{amsart}
\pdfoutput=1
\usepackage[utf8]{inputenc}

\usepackage{hyperref} 
\usepackage{enumerate}
\usepackage{graphicx}
\usepackage{mathtools}

\usepackage{amssymb,amsmath}

\usepackage{epsfig}  		
\usepackage{epic,eepic}       
\usepackage{datetime}

\let\epsilon=\varepsilon
\newcommand*{\Nset}{\mathbb{N}}  

\renewcommand*{\ker}{\operatorname{Ker}}

\newcommand*{\coim}{\operatorname{Coim}}

\newcommand*{\ol}{\overline}

\newcommand*{\id}{\operatorname{Id}}

\newcommand*{\der}{\mathrm{d}}
\newcommand*{\X}{\mathfrak{X}}
\DeclareMathOperator{\curl}{\operatorname{curl}}
\DeclareMathOperator{\Sol}{\operatorname{Sol}}

\newcommand*{\im}{\mathrm{Im}}

\newcommand*{\abs}[1]{\left\lvert#1\right\rvert}   
\newcommand*{\norm}[1]{\left\lVert#1\right\rVert}  
\newcommand*{\sis}[1]{\left\langle#1\right\rangle}  

\newcommand*{\sol}[1]{#1^\mathrm{s}}

\newcommand{\C}{\mathbb C}
\newcommand{\R}{\mathbb R}

\makeatletter
\def\blfootnote{\gdef\@thefnmark{}\@footnotetext}
\makeatother

\newtheorem{thm}{Theorem}[section]
\newtheorem{prop}[thm]{Proposition}
\newtheorem{lem}[thm]{Lemma}
\newtheorem{cor}[thm]{Corollary}
\theoremstyle{definition}

\theoremstyle{remark}
\newtheorem{remark}[thm]{Remark}
\numberwithin{equation}{section}

\usepackage{amsaddr}

\title[Mixed and transverse ray transforms]{On mixed and transverse ray transforms on orientable surfaces}

\author{Joonas Ilmavirta}
\address{Department of Mathematics and Statistics\\University of Jyv\"askyl\"a\\
P.O. Box 35 (MaD) FI-40014 University of Jyv\"askyl\"a, Finland\\
\texttt{joonas.ilmavirta@jyu.fi}}
\author{Keijo M\"onkk\"onen}
\address{Department of Mathematics and Statistics\\University of Jyv\"askyl\"a\\
P.O. Box 35 (MaD) FI-40014 University of Jyv\"askyl\"a, Finland\\
\texttt{kematamo@jyu.fi}}
\author{Jesse Railo}
\address{Department of Pure Mathematics and Mathematical Statistics\\ University of
Cambridge\\ Cambridge CB3 0WB, UK\\
\texttt{jr891@cam.ac.uk}}

\date{\today}

\begin{document}

\begin{abstract}
The geodesic ray transform, the mixed ray transform and the transverse ray transform of a tensor field on a surface can all be seen as what we call mixing ray transforms, compositions of the geodesic ray transform and an invertible linear map on tensor fields.
We provide an approach that uses a unifying concept of symmetry to merge various earlier transforms (including mixed, transverse, and light ray transforms) into a single family of integral transforms with similar kernels.
\end{abstract}





\maketitle


\section{Introduction}

We give an algebraic point of view to various geodesic ray transforms of tensor fields, unifying the Riemannian X-ray transform, the transverse ray transform, and the mixed ray transform, and the Lorentzian light ray transforms.
Our main result is a unifying point of view to two-dimensional ray transforms of tensor fields, not a single new injectivity result.
This approach comes with a natural notion of symmetry, which is not generally the same as the symmetry of the covariant tensor field whose integral transforms are under study, but arises from the structure of the relevant transform.

When two transforms differ from each other by a so-called mixing, they have the same injectivity properties by theorem~\ref{thm:linalg}.
Mixings turn mixed ray transforms into regular tensor transforms in two dimensions.
In corollary~\ref{cor:mixedraysimplesurface} we recast the injectivity result~\cite{dHSZ18} of the mixed ray transform on simple surfaces in our language and we provide a reproof in corollary~\ref{cor:sinjectivitymixed}.
These results are also extended to Cartan--Hadamard manifolds in corollary~\ref{cor:cartanhadamardinjectivity}.

The tensor tomography results~\cite{FIO-light-ray} on globally hyperbolic Lorentzian manifolds have a different kind of kernel than their Riemannian counterpart.
The kernel, when operating on symmetric tensor fields of order $m\geq2$, contains both potential fields and conformal multiples of the metric.
In the present approach the conformal gauge is absorbed into the concept of symmetry, making the statements of solenoidal injectivity (s-injectivity) fully analogous on Riemannian and Lorentzian manifolds; see corollary~\ref{cor:lorentzinjectivity}.

A number of corollaries of the method are given in this article, and we refrain from listing them all here.
Consequently we have a great amount of notation, and we have collected the key items in appendix~\ref{appendix} to help the reader.

\subsection{Mixing ray transforms}

Let~$M$ be a Riemannian manifold of dimension $n \geq 2$.
Let $f \in \X(T_mM)$ be a covariant $m$-tensor field (not necessarily symmetric) where $m\geq 1$.
We completely exclude the scalar case $m=0$ from our discussion.
Let $A\colon \X(T_mM) \to \X(T_mM)$ be an invertible linear map such that
\begin{equation}
(Af)_x(v_1,\dots,v_m)
=
f_x(A_1(x)v_1,\dots,A_m(x)v_m),
\end{equation}
where $A_i(x)\colon T_xM \to T_xM$ are linear isomorphisms.
The linear maps $\X(T_mM) \to \X(T_mM)$ of this form are called \emph{mixings} in this article.

We study the class of geodesic ray transforms, called \emph{mixing ray transforms}, defined by the formula
\begin{equation}
\label{eq:definition}
\begin{split}
I_Af(x,v)
&\coloneqq
\int_{\tau_-(x,v)}^{\tau_+(x,v)}
(Af)_{\gamma_{x,v}(t)}(\dot{\gamma}_{x,v}(t)^{\otimes m})
\der t
\\&=
\int_{\tau_-(x,v)}^{\tau_+(x,v)}
f_{\gamma_{x,v}(t)}(A_1(\gamma_{x,v}(t))\dot{\gamma}_{x,v}(t),\dots,A_m(\gamma_{x,v}(t))\dot{\gamma}_{x,v}(t))
\der t
,
\end{split}
\end{equation}
where $\gamma_{x,v}\colon [\tau_-(x,v), \tau_+(x,v)] \to M$ is the maximal unit speed geodesic through $(x,v) \in SM$. 
Formula~\eqref{eq:definition} is invariant under the geodesic flow $\varphi_t(x,v) = (\gamma_{x,v}(t),\dot{\gamma}_{x,v}(t))$, that is, $I_Af(x,v) = I_Af(\varphi_t(x,v))$ for any $t \in \R$ in the maximal domain of~$\gamma_{x,v}$.
This definition allows to define~$I_A$ on Riemannian manifolds without boundary, provided that the tensor field~$f$ is sufficiently integrable.
We remark that if $A_i= \id$ for every $i=1,\dots,m$, then~$I_A$ is the usual geodesic ray transform of tensor fields.
Other special cases of the mixing ray transforms in two dimensions have been studied earlier in \cite{BH-tomographic-reconstruction-vector-fields, dHSZ18, DS-tomography}, and somewhat related geodesic ray transforms in higher dimensions have been studied recently in \cite{ABH-support-theorem-transverse-ray, dESUZ-generic-uniqueness-mixed-ray, LW-diffraction-tomography}.
We remark that the mixing ray transforms are defined for all $n\geq 2$ but they do not include the higher dimensional transforms ($n\geq 3$) studied in \cite{ABH-support-theorem-transverse-ray, dESUZ-generic-uniqueness-mixed-ray, LW-diffraction-tomography}.

The main problems that we study are uniqueness and stability for recovering $f \in \X(T_mM)$ from the knowledge of~$I_Af$.
The main point of this work is an algebraic view of the mixing ray transforms.
We present many applications of the method and instead of having a main theorem we have a main idea how to study the mixing ray transforms.
We show in theorem~\ref{thm:linalg} and corollary~\ref{cor:sinjectivity} that the related inverse problems for~$I_A$ and~$I_{\widetilde{A}}$ with two different mixings~$A$ and~$\widetilde{A}$ can be reduced to each other.
Especially, this allows us to derive new uniqueness and stability results for the mixed and transverse ray transforms in two dimensions using the known results for the geodesic ray transform.
These results are given in corollaries~\ref{cor:sinjectivitymixed}, \ref{cor:cartanhadamardinjectivity}, \ref{cor:stability} and~\ref{cor:transverseinjectivity}.
Moreover, we show in corollaries~\ref{cor:simplesurfaceuniqueness} and~\ref{cor:cartanhadamarduniqueness} that on compact simple surfaces and on certain Cartan--Hadamard manifolds the geodesic ray transform and the transverse ray transform together determine one-forms uniquely.
This extends results in~\cite{BH-tomographic-reconstruction-vector-fields,DS-tomography} to more general Riemannian manifolds.

Furthermore, we study tensor decompositions and their symmetries with respect to these integral transforms.
These considerations lead us to corollaries~\ref{cor:mixedraysimplesurface} and~\ref{cor:lorentzinjectivity} which show how the earlier kernel characterizations of the mixed ray transform on compact simple surfaces and the light ray transform on static globally hyperbolic Lorentzian manifolds can be seen as s-injectivity results under the correct notions of symmetry.

\subsection{Related problems}

The geodesic ray transform has been studied extensively on Riemannian manifolds and s-injectivity is known in many cases.  For example, the geodesic ray transform is s-injective on tensor fields of any order on two-dimensional compact simple manifolds~\cite{PSU-tensor-tomography-on-simple-surfaces} and on simply connected compact manifolds with strictly convex boundary and non-positive curvature \cite{PS-sharp-stability-nonpositive-curvature, PS-integral-geometry-negative-curvature,SHA-integral-geometry-tensor-fields}. S-injectivity is also known on non-compact Cartan--Hadamard manifolds for all tensor fields which satisfy certain decay conditions \cite{LE-cartan-hadamard, LRS-tensor-tomography-cartan-hadamard}. We refer to the surveys \cite{IM:integral-geometry-review, PSU-tensor-tomography-progress} for a more comprehensive treatment of the geodesic ray transform and s-injectivity. The mixed ray transform has been studied mainly on two- and three-dimensional compact simple manifolds, and the kernel is known in these cases for a certain class of tensor fields \cite{dESUZ-generic-uniqueness-mixed-ray, dHSZ18} (see also~\cite{DS-tomography, SHA-integral-geometry-tensor-fields}). There are a few results for the transverse ray transform: in~$\R^2$ the kernel of the transverse ray transform on vector fields consists of curls of scalar fields~\cite{DS-tomography, NA-mathematical-methods-image-reconstruction}, and in higher dimensions the transform is even injective on certain manifolds~\cite{SHA-integral-geometry-tensor-fields} (see also ~\cite{ABH-support-theorem-transverse-ray} for a support theorem). Somewhat related transforms are the restricted transverse ray transform~\cite{VMS-transverse-partialdata} and the truncated transverse ray transform~\cite{LS-truncated-transverse, SHA-integral-geometry-tensor-fields}.

The usual applications of the geodesic ray transform are medical imaging \cite{NA-mathematics-computerized-tomography, NA-mathematical-methods-image-reconstruction}, Doppler tomography \cite{SCHU-importance-of-vector-field-tomography,SS-vector-field-overview} and seismic imaging \cite{SHA-integral-geometry-tensor-fields, SUVZ-travel-time-tomography}. The transverse ray transform has applications in polarization tomography~\cite{SHA-integral-geometry-tensor-fields}, photoelasticity~\cite{HL-applications-to-photoelasticity}, diffraction tomography~\cite{LW-diffraction-tomography} and also in the determination of the refractive index of gases \cite{BH-tomographic-reconstruction-vector-fields, SCHWA-flame-analysis-schlieren}. The mixed ray transform arises in seismology as a linearization of elastic travel time tomography problem \cite{dESUZ-generic-uniqueness-mixed-ray, SHA-integral-geometry-tensor-fields}.

\subsection*{Organization of the article} In section 2 we recall the preliminaries on the geodesic ray transform and the mixed ray transform.
In section 3 we define the mixing ray transforms and study their basic properties using an algebraic approach.
In section 4 we apply our methods for the mixed ray transform and the transverse ray transform on orientable two-dimensional Riemannian manifolds which admit s-injectivity of the geodesic ray transform. We have included some of our notation in appendix~\ref{appendix}.

\subsection*{Acknowledgements}
The authors wish to thank Teemu Saksala for helpful discussions related to the mixed ray transform.
J.I. was supported by Academy of Finland (grants 332890 and 336254).
K.M. and J.R. were supported by Academy of Finland (Centre of Excellence in Inverse Modelling and Imaging, grant numbers 284715 and 309963).

\section{Preliminaries}
\label{sec:preliminaries}

We mainly follow the reference~\cite{SHA-integral-geometry-tensor-fields} for the integral geometry part of this section.
Basic theory of differential geometry can be found in \cite{LEE-smooth-manifolds, LEE-riemannian-manifolds} and basic theory of Sobolev spaces of tensor fields on manifolds can be found for example in \cite{BH-sobolev-spaces-compact-manifolds, WE-uhlenbeck-compactness}.
We always assume that~$(M, g)$ is a connected Riemannian manifold, and we can sometimes allow it to be pseudo-Riemannian.

\subsection{Notation}\label{sec:notation}

If~$E$ is a vector bundle, we denote by~$\X(E)$ the space of all smooth sections of~$E$.
We use this notation whenever the regularity is unimportant.

We let~$T^{m_2}_{m_1}M=T^*M^{\otimes m_1}\otimes TM^{\otimes m_2}$ be the bundle of tensors of type $(m_2, m_1)$ over~$M$.
Then $\X(T^{m_2}_{m_1}M)$ is the space of all $(m_2, m_1)$-tensor fields on~$M$.
We also write $\X(T_m M)\coloneqq\X(T^0_m M)$.

We denote by $S_mM\subset \X(T_mM)$ the space of all symmetric covariant tensor fields.
When we want to emphasize the regularity of the tensor field, we replace~$\X$ with the regularity in question; for example $C^q(T_mM) \subset \X(T_mM)$, $q \in \Nset$, is the space of all $C^q$-smooth $(0, m)$-tensor fields on~$M$.
For symmetric tensor fields we write~$C^q(S_mM)$ and so on.
We use the Einstein summation convention, where every repeated index (both as a subscript and superscript) is implicitly summed over.

\subsection{Sobolev norms of tensor fields}\label{sec:sobolevnorms}

Let $f, h\in \X(T_mM)$ be tensor fields and $m \geq 1$. We define the fiberwise inner product as
\begin{equation}
g_x(f, h)=g^{i_1j_1}(x)\dotso g^{i_mj_m}(x)f_{i_1\dotso i_m}(x)h_{j_1\dotso j_m}(x)
\end{equation}
and the fiberwise norm is denoted by $\abs{f}_{g_x}=\sqrt{g_x(f, f)}$.
If $m=0$, we simply let $\abs{f}_{g_x}\coloneqq\abs{f(x)}$.

Let~$\der V_g(x)$ be the Riemannian volume measure on~$M$. If~$M$ is orientable, then~$\der V_g(x)$ is given by the Riemannian volume form and $\der V_g(x)=\sqrt{\det g(x)}\der x^1\wedge\dotso\wedge\der x^n$ where $(x^1,\dots,x^n)$ are any positively oriented smooth coordinates. We define the $L^p$-norm, $1\leq p<\infty$, of a tensor field $f \in \X(T_mM)$ by
\begin{equation}
\norm{f}_p=\bigg(\int_M\abs{f}_{g_x}^p \der V_g(x)\bigg)^{1/p}
\end{equation}
whenever the integral exists.

Denote by $\nabla^k f\in C^{q-k}(T_{m+k}M)$ the $k$th iterated covariant derivative of the tensor field $f \in C^q(T_mM)$ whenever $q \geq k \geq 0$ and $k, q \in \Nset$.
We define the Sobolev norm $\norm{\cdot}_{k, p}$ as
\begin{equation}
\norm{f}_{k, p}
=
\bigg(\sum_{i=0}^k\norm{\nabla^i f}_p^p\bigg)^{1/p}
\end{equation}
where $\nabla^0 f\coloneqq f$.
Let~$C^\infty_{k, p}(T_mM)$ be the set of smooth tensor fields~$f$ for which $\norm{f}_{k, p}<\infty$.
The Sobolev space~$W^{k, p}(T_mM)$ is defined to be the completion of~$C^\infty_{k, p}(T_mM)$ with respect to the norm~$\norm{\cdot}_{k, p}$.
We are mainly interested in the space $W^{k, 2}(T_mM)=:H^k(T_mM)$.
Then~$H^k(T_mM)$ is a Hilbert space with the inner product
\begin{equation}
\sis{f , h}_{H^k(T_mM)}
=
\sum_{i=0}^k\sis{\nabla^i f, \nabla^i h}_{L^2(T_{m+i}M)}
=
\sum_{i=0}^k\int_M g_x(\nabla^i f, \nabla^i h)\der V_g(x).
\end{equation}
Similarly one defines the Sobolev space $H^k(S_mM)\subset H^k(T_mM)$ as the completion of~$C^\infty_{k, 2}(S_mM)$ with respect to the norm induced by the inner product $\sis{\cdot, \cdot}_{H^k(T_mM)}$.

\subsection{Hodge star on orientable Riemannian surfaces}

Assume that $(M, g)$ is two-dimensional orientable Riemannian manifold.
For example,~$M$ is orientable if there is a smooth mapping $F\colon M\rightarrow N$ such that~$F$ is a local diffeomorphism and~$N$ is orientable, or if~$M$ is simply connected~\cite{LEE-smooth-manifolds}.
The Hodge star~$\star$ is an operator on one-forms $\star\colon \X(T_1M)\rightarrow \X(T_1M)$ and it corresponds to a 90 degree rotation counterclockwise.
Orientability of~$M$ guarantees that~$\star$ is a well-defined global operator.
Since we can identify one-forms with vector fields by the musical isomorphisms~$\flat$ and~$\sharp$, we can also rotate vector fields.
To shorten the notation, we simply let $\sharp \star \flat=:\star$ and locally we have
\begin{equation}
\star(v^1 e_1 + v^2 e_2) \coloneqq  -v^2 e_1 + v^1 e_2
\end{equation} 
in any positively oriented local orthonormal frame~$\{e_1,e_2\}$.


\subsection{The geodesic ray transform}
\label{subsec:geodesicraytransform}

For any set~$X$ we denote by~$\mathcal{F}(X)$ the space of all complex-valued functions $X\to \mathbb{C}$.
We define the map $\lambda\colon \X(T_mM) \rightarrow\mathcal{F}(TM)$ as
\begin{equation}
(\lambda f)(x, v) \coloneqq f_x(v, \dotso, v)=f_{i_1\dotso i_m}(x)v^{i_1}\dotso v^{i_m}
\end{equation}
where $f_{i_1\dotso i_m}(x)$ are the components of the tensor field $f\in \X(T_mM)$ in any local coordinates. 
We let $SM=\bigcup_{x\in M}S_xM$ be the sphere bundle where the fibers are the unit spheres $S_xM=\{v\in T_xM: \abs{v}_{g_x}=1\}$ of the tangent spaces~$T_xM$.
The unit sphere bundle~$SM$ is not to be confused with the space~$S_mM$ of symmetric covariant tensor fields of order~$m$.
The geodesic flow is defined as $\varphi_t(x, v)=(\gamma_{x, v}(t), \dot{\gamma}_{x, v}(t))$ where $\gamma_{x,v}(t)$ is the unique geodesic such that $(\gamma_{x,v}(0),\dot{\gamma}_{x,v}(0)) = (x,v) \in SM$.
If~$M$ has boundary~$\partial M$, we denote by $\tau(x, v)$ the first time when the geodesic $\gamma_{x,v}$ reaches~$\partial M$.

Assume that $(M, g)$ is compact and non-trapping Riemannian manifold with boundary. Non-trapping means that $\tau(x, v)<\infty$ for all $(x, v)\in SM$. We denote by $\partial_{\mathrm{in}} SM \subset \partial SM$ the inward-pointing unit vectors. We define the geodesic ray transform to be the operator $I: \mathcal{X} \to \mathcal{F}(\partial_{\mathrm{in}} SM)$ given by the formula
\begin{equation}
\label{eq:grtbdry}
If(x, v)=\int_0^{\tau(x, v)}(\lambda f)(\varphi_t(x, v))\der t, \quad (x, v)\in\partial_{\mathrm{in}} SM
\end{equation}
where $\mathcal{X} \subset \X(T_mM)$ is any set such that the integral in~\eqref{eq:grtbdry} is well-defined. Typically we choose $\mathcal{X} = C_c^\infty(T_mM)$ or $\mathcal{X} = H^k(T_mM)$. We note that two definitions~\eqref{eq:definition} and~\eqref{eq:grtbdry} agree when $A = \id$ and $(x,v) \in \partial_{\mathrm{in}} SM$ (in that case $\tau_+(x,v) = \tau(x,v)$ and $\tau_-(x,v) = 0$). One can also write $If=I_{SM}(\lambda f|_{SM})$ where the geodesic ray transform of a function $h\colon SM\rightarrow\R$ is
\begin{equation}
\label{eq:geodesicrayspherebundle}
I_{SM}h(x, v)=\int_0^{\tau(x, v)}h(\varphi_t(x, v))\der t, \quad (x, v)\in\partial_{\mathrm{in}} SM.
\end{equation}
One can then define an adjoint~$I^*$ by duality using an $L^2$-inner product. However, there are different measures on~$\partial_{\mathrm{in}} SM$ which lead to different adjoints.
We use the weighted measure defined in~\cite{PSU-tensor-tomography-progress} which is invariant under the scattering relation, and the normal operator $N=I^*I$ is defined with respect to this measure.

If $f\in H^k(S_mM)$ and~$M$ is a compact Riemannian manifold with boundary, then there is the solenoidal decomposition \cite[Theorem 3.3.2]{SHA-integral-geometry-tensor-fields}
\begin{equation}
f
=
\sol{f}+\sigma\nabla p, \quad \text{div}(\sol{f})=0, \ \ p|_{\partial M}=0
\end{equation}
where $\sol{f}\in H^k(S_mM)$, $p\in H^{k+1}(S_{m-1}M)$ and $m\geq 1$. Moreover, if $f \in C^\infty(S_mM)$, then $\sol{f}\in C^\infty(S_mM)$ and $p\in C^\infty(S_{m-1}M)$. Here~$\sigma$ is the symmetrization of tensor fields (see section~\ref{sec:decomp} for details) and~$\text{div}(\cdot)$ is the covariant divergence. The tensor field~$\sol{f}$ is the solenoidal part and~$\sigma\nabla p$ is the potential part of~$f$. By the fundamental theorem of calculus one sees that $I(\sigma\nabla p)=0$ since~$p$ vanishes on the boundary. Therefore potentials are always in the kernel of~$I$ and we can only try to recover the solenoidal part of~$f$ from~$I$. When $m\geq 1$ we say that~$I$ is solenoidally injective (s-injective) if for sufficiently regular $f\in S_mM$ it holds that $If=0$ if and only if $f=\sigma\nabla p$ for some (sufficiently regular) $p\in S_{m-1}M$ vanishing on the boundary. 

One particular class of manifolds where one usually studies the geodesic ray transform is the class of compact simple manifolds.
The manifold $(M, g)$ is simple if it is non-trapping, has no conjugate points and the boundary~$\partial M$ is strictly convex (the second fundamental form on~$\partial M$ is positive definite).
Each compact simple manifold is diffeomorphic to the Euclidean unit ball. It also follows that compact simple manifolds are simply connected and hence orientable \cite{PSU-tensor-tomography-on-simple-surfaces, THO-closed-geodesics}.

One can also study the geodesic ray transform on certain non-compact manifolds. The manifold $(M, g)$ without boundary is a Cartan--Hadamard manifold if it is complete, simply connected and its sectional curvature is nonpositive.
Cartan--Hadamard manifolds are always non-compact, orientable and diffeomorphic to~$\R^n$.
Basic examples of Cartan--Hadamard manifolds are Euclidean and hyperbolic spaces. On such manifolds the geodesic ray transform is defined as
\begin{equation}
\label{eq:geodesicraycartanhadamard}
If(x, v)=\int_{-\infty}^{\infty}(\lambda f)(\varphi_t(x, v))\der t, \quad (x, v)\in SM.
\end{equation}
Note that completeness implies that geodesics are defined on all times by the Hopf--Rinow theorem~\cite{LEE-riemannian-manifolds}.
We will use the following classes of tensor fields on Cartan--Hadamard manifolds:
\begin{equation}\label{eq:cartanhadamardspaces}
\begin{split}
E_{\eta}(T_mM)
&=
\{f\in C^1(T_mM):
\\&\qquad
\abs{f}_{g_x}\leq Ce^{-\eta d(x, o)} \ \text{for some} \ C>0\},
\\
E_{\eta}^1(T_mM)
&=
\{f\in C^1(T_mM):
\\&\qquad
\abs{f}_{g_x}+\abs{\nabla f }_{g_x}\leq Ce^{-\eta d(x, o)} \ \text{for some} \ C>0\},
\\
P_{\eta}(T_mM)&=\{f\in C^1(T_mM):
\\&\qquad
\abs{f}_{g_x}\leq C(1+d(x, o))^{-\eta} \ \text{for some} \ C>0\},
\\
P_{\eta}^1(T_mM)&=\{f\in C^1(T_mM):
\\&\qquad
\abs{f}_{g_x}\leq C(1+d(x, o))^{-\eta} \ \text{and}
\\&\qquad
\abs{\nabla f}_{g_x}\leq C(1+d(x, o))^{-\eta-1} \ \text{for some} \ C>0\}.
\end{split}
\end{equation}
Here $o\in M$ is fixed reference point and $\eta>0$.
The spaces defined above are independent of the choice of this point.

\subsection{The mixed and transverse ray transforms}
\label{sec:definitionofmixedray}
Define $S_kM\otimes S_lM\subset \X(T_{k+l}M)$ to be the set of $(k+l)$-tensor fields which are symmetric in the first~$k$ and last~$l$ variables.
Let~$S_k(T_xM)$ denote the space of symmetric $(0,k)$-tensors on~$T_xM$ for any fixed $x \in M$. If $f \in S_kM\otimes S_lM$, then $f_x \in S_k(T_xM) \otimes S_l(T_xM)$.
Let $\pi\colon\partial_{\mathrm{in}} SM\rightarrow M$ be the restriction of the projection of the tangent bundle.
Let~$\pi^*(S_kM)$ be the pullback bundle of symmetric $k$-tensor fields so that for every $\varphi\in \X(\pi^*(S_kM))$ and $(x, v)\in \partial_{\mathrm{in}} SM$ we have $\varphi_{x,v} \in S_k(T_xM)$.

Let $v \in S_xM$. We define the projection operator $p_v: T_x M \to v^\bot \subset T_xM$ as
\begin{equation}
p_v (w) \coloneqq  w - g_x(w,v)v = \left(\delta^{i}_{j}-v_{j}v^{i}\right)w^je_i
\end{equation} 
where the latter formula holds in any local coordinates. We then define the projection operator $P_{v}^k\colon S_k(T_xM)\rightarrow S_k(T_xM)$ by the formula
\begin{equation}
    (P_v^kh)(v_1,\dots,v_k) \coloneqq  h(p_v(v_1),\dots,p_v(v_k))
\end{equation}
for any $v_1,\dots,v_k \in T_xM$, and one can write in any local coordinates that
\begin{equation}
(P_{v}^kh)_{i_1\dotso i_k}=(\delta^{j_1}_{i_1}-v^{j_1}v_{i_1})\dotso(\delta^{j_k}_{i_k}-v^{j_k}v_{i_k})h_{j_1\dotso j_k}.
\end{equation}
We can identify~$p_v$ as a $(1,1)$-tensor by setting $\tilde{p}_v(\alpha, w) \coloneqq  \alpha(p_v(w))$ where $w \in T_xM$ and $\alpha \in T_x^*M$. We note that also $P_v^k h = \tilde{p}_v^{\otimes k} h$ where the product on the right hand side is a contraction of~$\tilde{p}_v^{\otimes k}$ by~$h$.

We define the contraction of $f \in S_k(T_xM)\otimes S_l(T_xM)$ by $v \in T_xM$ with respect to the last~$l$ arguments as a mapping $\Lambda_{v}^l\colon S_k(T_xM)\otimes S_l(T_xM)\rightarrow S_k(T_xM)$ by
\begin{equation}
(\Lambda_{v}^lf)_{i_1\dotso i_k}=f_{i_1\dotso i_k j_1\dotso j_l}v^{j_1}\dotso v^{j_l}.
\end{equation}
Let us denote by $\mathcal{T}^{t\to s}_{\gamma}$ the parallel transport along~$\gamma$ from~$\gamma(t)$ to~$\gamma(s)$ whenever $s, t \in \R$ belong to the maximal domain of~$\gamma$. The mixed ray transform is the map $L_{k, l}\colon S_kM\otimes S_lM\rightarrow \pi^*(S_kM)$ defined as
\begin{equation}
L_{k, l}f(x, v)
\coloneqq 
\int_{0}^{\tau(x, v)}
\mathcal{T}^{t\to0}_{\gamma_{x,v}}(P_{\dot{\gamma}_{x,v}(t)}^k\Lambda_{\dot{\gamma}_{x,v}(t)}^lf_{\gamma_{x,v}(t)})
\der t
\end{equation}
for any $(x,v) \in \partial_{\mathrm{in}} SM$,
whenever the integral is well-defined.
We note that
\begin{equation}
\label{eq:Plambdaf}
\begin{split}
&
P_{\dot{\gamma}_{x,v}(t)}^k\Lambda_{\dot{\gamma}_{x,v}(t)}^lf_{\gamma_{x,v}(t)}(w_1,\dots,w_k)
\\&=
f_{\gamma_{x,v}(t)}(p_{\dot{\gamma}_{x,v}(t)}w_1,\dots,p_{\dot{\gamma}_{x,v}(t)}w_k,\dot{\gamma}_{x,v}(t),\dots,\dot{\gamma}_{x,v}(t))
\end{split}
\end{equation}
for any $w_1,\dots,w_k \in T_{\gamma_{x,v}(t)}M$.

Using~\eqref{eq:Plambdaf} and the definition of the parallel transport~$\mathcal{T}^{t\to s}_{\gamma}$, one can show that the mixed ray transform acts on $(x,v) \in \partial_{\mathrm{in}} SM$ as
\begin{equation}
\label{eq:mixduality}
\begin{split}
&\sis{L_{k, l}f(x,v), (\eta+av)^{\otimes k}}
\\&=
\int_0^{\tau(x, v)} f_{i_1\dotso i_k j_1\dotso j_l}(\gamma_{x, v}(t))\eta^{i_1}_{x, v}(t)\dotso\eta_{x, v}^{i_k}(t)\dot{\gamma}_{x, v}^{j_1}(t)\dotso\dot{\gamma}_{x, v}^{j_l}(t)\der t
\end{split}
\end{equation}
where $(\eta+av)^{\otimes k}$ is the tensor product of $\eta+av$ with itself~$k$ times, $a\in\R$ and $\eta_{x, v}(t)$ is the parallel transport of a vector $\eta=\eta_{x, v}(0) \in T_xM$ orthogonal to $v=\dot{\gamma}_{x, v}(0)$, see \cite[Chapter 5.2]{SHA-integral-geometry-tensor-fields} for details. 

The mixed ray transform is considerably simpler when~$M$ is orientable and $n=2$. Then~$v^{\perp}$ is one-dimensional for all $(x, v)\in\partial_{\mathrm{in}} SM$ and there is only one possible choice (modulo sign) for the vector~$\eta$ which is parallel transported along~$\gamma$.
We choose the orthogonal vector field as $\eta(t)=(\star\dot{\gamma})(t)$.
It is clear that $\star\dot{\gamma}\perp\dot{\gamma}$ at every point on the geodesic~$\gamma$ and that $D_t^{\gamma}(\star\dot{\gamma})=0$ where~$D_t^{\gamma}$ is the covariant derivative along the geodesic~$\gamma$.
Therefore~$\star\dot{\gamma}$ is parallel along~$\gamma$. Now using formula~\eqref{eq:mixduality} the mixed ray transform can be seen as a composition $L_{k, l}=I\circ A_{k, l}$ where 
\begin{equation}
\label{eq:mixedraytwodimensional}
(A_{k, l}f)_x(v_1, \dotso v_m)=f_x(A_1v_1, \dotso , A_mv_m)
\end{equation}
and $A_i = \star$ when $i = 1,\dots, k$ and $A_i = \id$ when $i = k+1,\dots, k+l$. Thus in two dimensions the mixed ray transform operates as
\begin{equation}
\label{eq:mixedraydefinition}
L_{k, l}f(x, v)=\int_0^{\tau(x, v)}(\lambda (A_{k, l}f))(\varphi_t(x, v))\der t, \quad (x, v)\in\partial_{\mathrm{in}} SM,
\end{equation}
and with these choices of~$A_{k, l}$ we have $L_{k, l}=I_{A_{k, l}}$ where the transform~$I_{A_{k, l}}$ is given by formula~\eqref{eq:definition}. 
If $k=0$, then~$L_{0, l}$ reduces to the geodesic ray transform~$I$.
If $l=0$, we call~$L_{k, 0}$ the transverse ray transform and use the notation $I_{\perp}\coloneqq L_{k, 0}$.
In higher dimensions $n>2$ the operator~$\star$ cannot be used to define the mixed ray transform since it maps $k$-forms into $(n-k)$-forms.

\section{The algebraic structure of mixing ray transforms}
\label{sec:linkku}

\subsection{Decompositions of tensor fields}
\label{sec:decomp}
Let $\sigma\colon \X(T_mM)\rightarrow S_mM$ be the usual symmetrization map of tensor fields where $m\geq 2$. We remind that if $m = 1$, then any $f \in \X(T_mM)$ is symmetric.
The components of~$\sigma f$ at a point $x\in M$ are 
\begin{equation}\label{eq:sigmadef}
(\sigma f)_{i_1\dotso i_m}(x)=\frac{1}{m!}\sum_{\tau\in\Pi_m}f_{i_{\tau(1)}\dotso i_{\tau(m)}}(x)
\end{equation}
where~$\Pi_m$ is the group of permutations.
The symmetrization~$\sigma$ is a projection $\X(T_mM) \to S_mM$, and it turns out to be orthogonal at every point with respect to any Riemannian metric by proposition~\ref{prop:V1V2V3}.
In particular,~$\sigma$ is idempotent and we can decompose the space~$\X(T_mM)$ as
\begin{equation}
\label{eq:decompositionsymmetrization}
\ker(\sigma)\oplus \im(\sigma)=\X(T_mM)
\end{equation}
by letting $f=(f-\sigma f)+\sigma f$. The decomposition~\eqref{eq:decompositionsymmetrization} can be done on any differentiable manifold~$M$. The set~$\ker(\sigma)$ can be identified with antisymmetric tensor fields when $m=2$ and
for $m>2$ the antisymmetric tensor fields are a strict subset of~$\ker(\sigma)$.

Recall that the map $\lambda\colon \X(T_mM)\rightarrow \mathcal{F}(TM)$ was defined as
\begin{equation}\label{eq:lambdaoperator}
(\lambda f)(x, v)=f_x(v, \dotso, v)
\end{equation}
where~$\mathcal{F}(TM)$ is the space of all complex-valued functions on~$TM$. We note that the restriction~$\lambda f|_{SM}$ determines~$\lambda f$ completely since~$f_x$ is homogeneous of degree~$m$. It follows directly from the definitions that $\lambda\circ\sigma=\lambda$. It is true that $\ker(\sigma)=\ker(\lambda)$ (see proposition~\ref{prop:V1V2V3}) and $\ker(\lambda) \subset\ker(I)$. Hence we call $\ker(\lambda) \subset \X(T_mM)$ the set of \emph{$\lambda$-antisymmetric tensor fields} or \emph{trivial part of the kernel of the geodesic ray transform} depending on the context.

We denote by $\lambda_x\colon (T^*_xM)^{\otimes m}\rightarrow \mathcal{F}(T_xM)$ the map $(\lambda_x \omega)(v)=\omega (v, \dotso, v)$, i.e. $(\lambda f)(x, v)=(\lambda_x f_x)(v)$. We let~$\sigma_x$ be the symmetrization of $m$-tensors in $(T^*_xM)^{\otimes m}$ and $S_m(T_xM)$ is the space of symmetric $m$-tensors in~$(T^*_xM)^{\otimes m}$.
We have the following proposition which summarizes some important connections between the different concepts introduced above.

\begin{prop}
\label{prop:V1V2V3}
Suppose that $m \geq 2$ and let~$M$ be a Riemannian (or pseudo-Riemannian) manifold. Let $x \in M$ and define the sets
\begin{enumerate}[(a)]
\item $V_1 = S_m(T_xM)$, $V_2 = \im(\sigma_x)$ and $V_3 = \ker(\lambda_x)^\bot$.\label{item:ekaprop}
\item $W_1 = (S_m(T_xM))^\bot$, $W_2 = \ker(\sigma_x)$ and $W_3 = \ker(\lambda_x)$.
\end{enumerate}
Then $V_1 = V_2 = V_3$, $W_1 = W_2 = W_3$, and $V_i \oplus W_j = (T_x^*M)^{\otimes m}$ for any $i, j =1,2,3$.
\end{prop}

\begin{proof}
It follows directly from the definitions that $V_1 = V_2$. Suppose that $W_2 = W_3$ and $V_3 \subset V_1$. This implies that 
\begin{equation}(T_x^*M)^{\otimes m} = V_2 \oplus W_2= V_3 \oplus W_3 = V_3 \oplus W_2.\end{equation} Since $V_3 \subset V_1 = V_2$, we get that $V_2 = V_3$. It then follows that $V_1 = V_2 = V_3$, $W_1 = W_2 = W_3$, and $V_i \oplus W_j = (T_x^*M)^{\otimes m}$ for any $i, j =1,2,3$. Hence it is sufficient to show that $W_2 = W_3$ and $V_3 \subset V_1$.

Let us first prove that $W_2 = W_3$. It is clear that $\ker(\sigma_x) \subset \ker(\lambda_x)$ since $\lambda_x \circ \sigma_x = \lambda_x$. Let $f \in \ker(\lambda_x)$. It now follows that $\sigma_x f \in \ker(\lambda_x)$. The polarization identity for symmetric multilinear maps \cite[Theorem 1]{THO-polarization-identity} states that a symmetric multilinear map is uniquely determined by its restriction to the diagonal. Since $\lambda_x \sigma_x f$ is the restriction of $\sigma_x f\colon (T_xM)
^m\to \C$ to the diagonal of $(T_xM)^m$, $\sigma_x f \in S_m(T_xM)$ and $\lambda_x \sigma_x f = 0$, we obtain that $\sigma_x f = 0$. This shows that $\ker(\lambda_x) \subset \ker(\sigma_x)$, and we conclude that $W_2 = W_3$.

Let us then prove that $V_3 \subset V_1$. Let $f \in V_3$. Fix some indices $j_1^\prime, \dotso , j_m^\prime$ and define the components of the tensor $h\in (T_x^*M)^{\otimes m}$ as
\begin{equation}
h_{j_1\dotso j_m}=\big(\delta^{j_1^\prime}_{j_1}\delta^{j_m^\prime}_{j_m}-\delta^{j_1^\prime}_{j_m}\delta^{j_m^\prime}_{j_1}\big)\delta^{j_2^\prime}_{j_2}\dotsm\delta^{j_{m-1}^\prime}_{j_{m-1}}.
\end{equation}
Then $h\in\ker(\lambda_x)$ and $g_x(f, h)=0$ implies that $f_{i_1\dotso i_m}=f_{i_m\dotso i_1}$.
By switching the order of the indices~$j_{k}$ in the definition of~$h$ one sees that~$f_{i_1\dotso i_m}$ has to be symmetric with respect to all indices. Hence $V_3\subset V_1$. This completes the proof.
\end{proof}

\begin{remark} This gives a somewhat unintuitive implication that the orthogonal complement of~$S_m(T_xM)$ does not depend on the Riemannian metric~$g_x$ at the point $x \in M$. This follows from proposition~\ref{prop:V1V2V3} since the mapping~$\sigma_x$ does not depend on~$g_x$. Proposition~\ref{prop:V1V2V3} also shows that the symmetrization~$\sigma\colon\X(T_mM) \to S_mM$ is an orthogonal projection when~$M$ is equipped with any Riemannian or pseudo-Riemannian metric.
\end{remark}

By proposition~\ref{prop:V1V2V3} we have many choices for the decomposition of the space~$(T^*_xM)^{\otimes m}$. We will use the orthogonal complement so that $\ker(\lambda_x)\oplus\ker(\lambda_x)^\perp=(T^*_xM)^{\otimes m}$. This allows us to decompose the space~$\X(T_mM)$ in the following way. We define the space~$\ker(\lambda)^\perp$ by saying that $f\in\ker(\lambda)^{\perp}$ if and only if $f_x\in\ker(\lambda_x)
^\perp$ for all $x\in M$. Define the projection $\widehat{\sigma}\colon\X(T_mM)\rightarrow\ker(\lambda)^\perp$ such that $(\widehat{\sigma}f)_x=P_{\ker(\lambda_x)
^\perp}f_x$ where $P_{\ker(\lambda_x)^\perp}$ is the orthogonal projection $P_{\ker(\lambda_x)^\perp}\colon (T^*_xM)^{\otimes m}\rightarrow\ker(\lambda_x)^\perp$. Then $f=(f-\widehat{\sigma}f)+\widehat{\sigma}f$ where $f-\widehat{\sigma}f\in\ker(\lambda)$ and $\widehat{\sigma}f\in\ker(\lambda)^\perp$. Hence we have the orthogonal decomposition
\begin{equation}
\ker(\lambda)\oplus\ker(\lambda)^\perp=\X(T_mM)
\end{equation}
where orthogonality is understood pointwise. We call the map~$\widehat{\sigma}$ a \emph{$\lambda$-symmetrization}. Note that $\ker(\lambda)=\ker(\sigma)$ and $\ker(\lambda)^\perp=S_mM$ by proposition~\ref{prop:V1V2V3}.

Another way to view $\lambda$-symmetric tensor fields is to take the quotient space $\coim(\lambda)=\X(T_mM)/{\ker(\lambda)}$ which identifies all tensor fields which differ by an element of~$\ker(\lambda)$.
This definition is natural for the geodesic ray transform in the sense that $If=Ih$ whenever $f\sim h$.
It follows that if~$V$ is any algebraic complement of~$\ker(\lambda)$, i.e. $\ker(\lambda)\oplus V=\X(T_mM)$, then $V\cong \coim(\lambda)$ via the map $v\mapsto [v]$ where~$[v]$ is the equivalence class of~$v$. This shows that one can realize the abstract quotient space~$\coim(\lambda)$ as a complementary subspace of~$\ker(\lambda)$ and that all complementary subspaces are isomorphic. 

More generally, let $\Omega=\bigcup_{x\in M}\Omega_x\subset TM$ where $\Omega_x\subset T_xM$. Let~$r_x$ be the restriction of a multilinear map on~$T_xM$ to~$\Omega_x$. As before we can decompose $(T^*_xM)^{\otimes m}=\ker(\lambda_{r, x})\oplus\ker(\lambda_{r, x})^\perp$ where $\lambda_{r, x}=r_x\circ\lambda_x$. This splitting can be done globally as follows. Denote by $r\colon\mathcal{F}(TM)\rightarrow\mathcal{F}(\Omega)$ the restriction to~$\Omega$ and define $\lambda_r=r\circ\lambda$. Then we have the decomposition
\begin{equation}\label{eq:decompositionlambdar}
\ker(\lambda_r)\oplus \ker(\lambda_r)^\perp=\X(T_mM)
\end{equation}
by writing $f=(f-\widehat{\sigma}_rf)+\widehat{\sigma}_rf$ where $\widehat{\sigma}_r\colon\X(T_mM)\rightarrow\ker(\lambda_r)^\perp$ is defined as $(\widehat{\sigma}_r f)_x=P_{\ker(\lambda_{r, x})^\perp}f_x$ and the space~$\ker(\lambda_r)^\perp$ is defined pointwise as earlier. We call the projection~$\widehat{\sigma}_r$ a \emph{$\lambda_r$-symmetrization}. As above~$\ker(\lambda_r)^\perp$ can be seen as a realization of the quotient space $\coim(\lambda_r)=\X(T_mM)/\ker(\lambda_r)$. It follows that $\ker(\lambda)\subset\ker(\lambda_r)$ and $\ker(\lambda_r)
^\perp\subset\ker(\lambda)^\perp\subset S_mM$ by proposition~\ref{prop:V1V2V3}.

Note that if~$r$ is the restriction to~$SM$, then $\ker(\lambda_r)=\ker(\lambda)$.
The geodesic ray transform can then be seen as a composition $I=I_{SM}\circ\lambda_r$.
We will generalize this approach in the next subsection.

\subsection{The mixing ray transform}
\label{sec:mixingraydefinition}
Let~$\operatorname{Aut}(TM)$ be the automorphism bundle of the tangent bundle.
A section~$B$ of this bundle, called an automorphism field, is a field whose value~$B(x)$ at any $x\in M$ is an automorphism (a linear self-bijection) of~$T_xM$.
In local coordinates~$B$ can be expressed as
\begin{equation}\label{eq:automorphism}
B(x)
=
B^j_k(x)\der x^k\otimes\partial_j
\end{equation}
where~$B^j_k(x)$ is an invertible matrix at every point~$x$.

Let $A_i$, $i=1,\dots,m$, be smooth automorphism fields.
Their tensor product $A=A_1\otimes\dots\otimes A_m$ is a mapping of tensor fields, $A\colon \X(T_mM)\rightarrow \X(T_mM)$.
From an invariant point of view it operates on a tensor field $f\in\X(T_mM)$ as
\begin{equation}\label{eq:mixingdefinition}
(Af)_x(v_1,\dots,v_m) = f_x(A_1(x)v_1,\dots,A_m(x)v_m),
\end{equation}
and it can be written in local coordinates as
\begin{equation}
(Af)_{i_1\dotso i_m}(x)=(A_1)^{j_1}_{i_1}(x)\dotso (A_m)^{j_m}_{i_m}(x)f_{j_1\dotso j_m}(x).
\end{equation}
Since each~$A_i(x)$ (or~$(A_i)^j_k(x)$) is invertible and~$A_i$ is smooth also~$A$ is invertible and smooth.
We call such map~$A$ an \emph{admissible mixing of degree~$m$}.

Let $r\colon\mathcal{F}(TM)\rightarrow\mathcal{F}(\Omega)$ be the restriction to $\Omega=\bigcup_{x\in M}\Omega_x\subset TM$ where $\Omega_x\subset T_xM$ and $\lambda_r=r\circ\lambda$. Let~$Z$ be a vector space and $J\colon \mathcal{F}(\Omega) \to Z$ a linear mapping.
We define the abstract ray transform $I_{A, r}\colon \X(T_mM)\rightarrow Z$ as
\begin{equation}
\label{eq:abstractmixingraytransform}
I_{A,r} = J \circ \lambda_r \circ A.
\end{equation}
Usually~$r$ is the restriction to~$SM$ and~$J$ is the geodesic ray transform on~$SM$. We call~$I_{A,r}$ \emph{the mixing ray transform} when these assumptions hold and write $\lambda \coloneqq \lambda_r$ and $I_A\coloneqq I_{A, r}$ to simplify our notation.
Next, we decompose the space~$\X(T_mM)$ into symmetric and antisymmetric parts with respect to $\lambda_r\circ A$.
Assume we have the decomposition $\ker(\lambda_r)\oplus \ker(\lambda_r)^\perp=\X(T_mM)$.
Since~$A$ is bijective linear map, we have $\ker(\lambda_r\circ A)=A^{-1}(\ker(\lambda_r))$ and
\begin{equation}\label{eq:decompositionlambdarandA}
\begin{split}
\X(T_mM)
&=
A^{-1}(\ker(\lambda_r)\oplus \ker(\lambda_r)^\perp)
\\&=
\ker(\lambda_r\circ A)\oplus A^{-1}(\ker(\lambda_r)^\perp).
\end{split}
\end{equation}
Hence we choose $A^{-1}(\ker(\lambda_r)^\perp)$ as the space of $(\lambda_r\circ A)$-symmetric tensor fields. The symmetrization map~$\widehat{\sigma}_{A, r}$ is a projection onto $A^{-1}(\ker(\lambda_r)^\perp)$ and it has the expression $\widehat{\sigma}_{A, r}=A^{-1}\circ\widehat{\sigma}_r\circ A$ where~$\widehat{\sigma}_r$ is a projection onto~$\ker(\lambda_r)^\perp$. 

One can also naturally define the mixing ray transform on a quotient space as a mapping
\begin{equation}
    I^q_{A, r}\colon \X(T_mM)/\ker(\lambda_r\circ A)\rightarrow Z
\end{equation}
such that
\begin{equation}
\label{eq:quotienttransform}
I_{A, r}^q [f]=I_{A, r}f,
\end{equation}
where $[f]\in\X(T_mM)/\ker(\lambda_r\circ A)$ is the equivalence class of $f\in\X(T_mM)$.
The transform~$I^q_{A, r}$ is well-defined, i.e. it does not depend on the representative.

We conclude this subsection with the following theorem which basically says that it is enough to know the properties of one mixing ray transform since any other mixing ray transform can be reduced to the known case.

\begin{thm}
\label{thm:linalg}
Let $E_1, E_2, E_3\subset\X(T_mM)$ be subspaces and $m\geq 1$. Assume that~$A$ and~$\widetilde{A}$ are admissible mixings of degree~$m$ and let $\mathcal{D}= A^{-1} \circ \widetilde{A}$. Then the following properties hold:
\begin{enumerate}[(a)]
\item \label{item:kernelcharacterization}
\textbf{Kernel characterization:}
Let $\mathcal{H}=\id-\widehat{\sigma}_{A, r}$ and $Y=\ker(I_{\widetilde{A}, r})\cap \widetilde{A}
^{-1}(\ker(\lambda_r)^\perp)$.
Then $f\in\ker(I_{A, r})$ if and only if 
$f=\mathcal{H}f+\mathcal{D}w$ for some $w\in Y$.
We have the decomposition
\begin{equation}
\ker(I_{A, r})=\im(\mathcal{H})\oplus\im(\mathcal{D}|_Y)=\ker(\lambda_r\circ A)\oplus\im(\mathcal{D}|_Y)
\end{equation}
where $\ker(I_{A, r}), \im(\mathcal{H}), \im(\mathcal{D}|_Y)\subset \X(T_mM)$.
\smallskip
\item\label{item:reconstruction}
\textbf{Reconstruction:}
Let $\mathcal{R}_{\widetilde{A},r}\colon Z \to S$ be a left inverse of $I_{\widetilde{A},r}\colon S \to Z$, where $S \subset \widetilde{A}^{-1}(\ker(\lambda_r)^\perp)$.
Then $\mathcal{R}_{A,r}= \mathcal{D} \circ \mathcal{R}_{\widetilde{A},r}\colon Z \to \mathcal{D}(S)$ is a left inverse of $I_{A,r}\colon \mathcal{D}(S) \to Z$ where $\mathcal{D}(S) \subset A^{-1}(\ker(\lambda_r)^\perp)$.
\smallskip
\item\label{item:stabilitygeodesictransform}
\textbf{Stability:}
Let $(Z,\norm{\cdot}_Z)$ and $(E_1,\norm{\cdot}_{E_1})$ be normed spaces.
Also assume that~$\mathcal{D}$ is bounded on $(E_1, \norm{\cdot}_{E_1})$ and that the estimate $\norm{f}_{E_1} \leq C\norm{I_{\widetilde{A},r}f}_Z$ holds for some subset $S^\prime\subset \widetilde{A}^{-1}(\ker(\lambda_r)^\perp)$.
Then the estimate 
\begin{equation}
\norm{f}_{E_1} \leq C\norm{\mathcal{D}}_{E_1}\norm{I_{A,r}f}_Z
\end{equation}
holds for all $f \in \mathcal{D}(S^\prime)\subset A^{-1}(\ker(\lambda_r)^\perp)$. 
\smallskip
\item \label{item:adjoint}
\textbf{Adjoint and normal operator:}
Let $(Z, \sis{\cdot, \cdot}_Z)$ and $(E_2, \sis{\cdot, \cdot}_{E_2})$ be Hilbert spaces. 
Assume that~$\mathcal{D}^{-1}$ is bounded in $(E_2, \sis{\cdot, \cdot}_{E_2})$ and that $I_{\widetilde{A},r}\colon E_2 \to Z$ is bounded.
Then the adjoints and the normal operators of~$I_{A,r}$ and~$I_{\widetilde{A},r}$ satisfy the formulas
\begin{equation}
    I_{A,r}^* = (\mathcal{D}^{-1})^*I_{\widetilde{A},r}^*,
    \quad
    N_{A,r} = (\mathcal{D}^{-1})^*N_{\widetilde{A},r}\mathcal{D}^{-1}.
\end{equation}
\smallskip
\item\label{item:stabilitynormaloperators}
\textbf{Stability with normal operators:}
Suppose that the assumptions of~(\ref{item:adjoint}) hold and let $\norm{\cdot}_{E_3}$ be a norm on~$E_3$.
Assume also that~$\mathcal{D}^*$ is bounded in $(E_3, \norm{\cdot}_{E_3})$ and that the estimate $\norm{f}_{E_2} \leq C\norm{N_{\widetilde{A},r}f}_{E_3}$ holds for some subset $S^{\prime\prime}\subset\widetilde{A}^{-1}(\ker(\lambda_r)^\perp)$.
Then the estimate
\begin{equation}
    \norm{f}_{E_2} \leq C\norm{\mathcal{D}}_{E_2}\norm{\mathcal{D}^*}_{E_3} \norm{N_{A,r}f}_{E_3}
\end{equation}
holds for all $f \in \mathcal{D}(S^{\prime\prime})\subset A^{-1}(\ker(\lambda_r)^\perp)$.
\end{enumerate}
\end{thm}

\begin{proof}
(\ref{item:kernelcharacterization})
If~$f$ is of the form $f=\mathcal{H}f+\mathcal{D}w$ for some $w\in Y$, then clearly $f\in\ker(I_{A, r})$.
For the converse, let $w=\widehat{\sigma}_{\widetilde{A}, r}\mathcal{D}^{-1}f=\mathcal{D}^{-1}\widehat{\sigma}_{A, r}f$.
We can write 
\begin{equation}
f=(f-\widehat{\sigma}_{A, r}f)+\widehat{\sigma}_{A, r}f=\mathcal{H}f+\mathcal{D}\mathcal{D}^{-1}\widehat{\sigma}_{A, r}f=\mathcal{H}f+\mathcal{D}w.
\end{equation}
Clearly $w\in\widetilde{A}^{-1}(\ker(\lambda_r)^\perp)$ and 
\begin{equation}
I_{\widetilde{A}, r}w=I_{\widetilde{A}, r}\mathcal{D}^{-1}\widehat{\sigma}_{A, r}f=I_{A, r}\widehat{\sigma}_{A, r}f=I_{A, r}f=0
\end{equation}
so $w\in\ker(I_{\widetilde{A}, r})$.
Assume then that $f\in\im(\mathcal{H})\cap\im(\mathcal{D}|_Y)$.
Now $f\in\ker(\lambda_r\circ A)$ and $f=\mathcal{D}w$ where $w\in\widetilde{A}^{-1}(\ker(\lambda_r)^\perp)$.
But this implies $w=\mathcal{D}^{-1}f\in\ker(\lambda_r\circ\widetilde{A})$ and hence $w=0$.
Therefore $\im(\mathcal{H})\cap\im(\mathcal{D}|_Y)=\{0\}$.
\medskip

(\ref{item:reconstruction})
Clearly $\mathcal{D}(S)\subset\mathcal{D}(\widetilde{A}^{-1}(\ker(\lambda_r)^\perp))=A^{-1}(\ker(\lambda_r)^\perp)$.
Let $f\in\mathcal{D}(S)$.
Then 
\begin{equation}
\mathcal{D}\mathcal{R}_{\widetilde{A}, r}I_{A, r}f=\mathcal{D}\mathcal{R}_{\widetilde{A}, r}I_{\widetilde{A}, r}\mathcal{D}^{-1}f=\mathcal{D}\mathcal{D}^{-1}f=f
\end{equation}
implying that $\mathcal{D}\circ\mathcal{R}_{\widetilde{A}, r}$ is a left inverse of~$I_{A, r}$ on~$\mathcal{D}(S)$.
\medskip

(\ref{item:stabilitygeodesictransform})
For $f\in\mathcal{D}(S^\prime)$ we find
\begin{align}
\norm{f}_{E_1}&=\norm{\mathcal{D}\mathcal{D}^{-1}f}_{E_1}\leq\norm{\mathcal{D}}_{E_1}\norm{\mathcal{D}^{-1}f}_{E_1}\leq C\norm{\mathcal{D}}_{E_1}\norm{I_{\widetilde{A}, r}\mathcal{D}^{-1}f}_Z \\
&=C\norm{\mathcal{D}}_{E_1}\norm{I_{A, r}f}_Z
\end{align}
\medskip
as claimed.

(\ref{item:adjoint})
Using the definitions of adjoints, we obtain
\begin{align}
\sis{I_{A, r}f, h}_Z=\sis{I_{\widetilde{A}, r}\mathcal{D}^{-1}f, h}_Z=\sis{\mathcal{D}^{-1}f, I_{\widetilde{A}, r}^*h}_{E_2}=\sis{f, (\mathcal{D}^{-1})^*I_{\widetilde{A}, r}^* h}_{E_2}.
\end{align}
Hence $I_{A, r}^*=(\mathcal{D}^{-1})^*I_{\widetilde{A}, r}^*$ and the normal operator becomes 
\begin{equation}
N_{A, r}=I_{A, r}^*I_{A, r}=(\mathcal{D}^{-1})^*I_{\widetilde{A}, r}^*I_{\widetilde{A}, r}\mathcal{D}^{-1}=(\mathcal{D}^{-1})^* N_{\widetilde{A}, r}\mathcal{D}^{-1}.
\end{equation}

(\ref{item:stabilitynormaloperators})
If $f\in\mathcal{D}(S^{\prime\prime})$, then we have
\begin{align}
\norm{f}_{E_2}&\leq \norm{\mathcal{D}}_{E_2}\norm{\mathcal{D}^{-1}f}_{E_2}\leq C\norm{\mathcal{D}}_{E_2}\norm{N_{\widetilde{A}, r}\mathcal{D}^{-1}f}_{E_3} \\
&=C\norm{\mathcal{D}}_{E_2}\norm{\mathcal{D}^*(\mathcal{D}^{-1})^*N_{\widetilde{A}, r}\mathcal{D}^{-1}f}_{E_3} \\
&\leq C\norm{\mathcal{D}}_{E_2}\norm{\mathcal{D}^*}_{E_3}\norm{N_{A, r}f}_{E_3}. \qedhere
\end{align}
\end{proof}

\subsection{Solenoidal injectivity}\label{sec:solenoidalinjectivity}

In this section we analyze more closely the kernel characterization given in theorem~\ref{thm:linalg}(\ref{item:kernelcharacterization}). Specifically, we apply our methods to show s-injectivity of general mixing ray transforms when s-injectivity of the geodesic ray transform is known. We also use our approach to show that the earlier results about the kernel of the mixed ray transform on compact simple surfaces and the kernel of the light ray transform on static globally hyperbolic Lorentzian manifolds can be seen as s-injectivity results under correct notions of symmetry.

\subsubsection{General results}
\label{subsubsection:generalsinjectivity}
Let~$r$ be the restriction to~$SM$ and $J=I_{SM}$ so that $I=I_{SM}\circ\lambda_r$ and $I_{A, r}=I\circ A$. Since now $\ker(\lambda_r)=\ker(\lambda)$ we use an abuse of notation and write $\lambda\coloneqq \lambda_r$.
By proposition~\ref{prop:V1V2V3} we can choose $\ker(\lambda)^\perp=S_mM$ so that $\widehat{\sigma}_A=A^{-1}\circ\sigma\circ A$ is a projection onto~$A^{-1}(S_mM)$.
Further, we define the covariant derivative $\nabla^A=A^{-1}\circ\nabla$.
The derivative~$\nabla^A$ is natural for the transform~$I_A$ since if $v|_{\partial M}=0$, then $I_A(\widehat{\sigma}_A\nabla^Av)=0$.

We say that the mixing ray transform~$I_A$ is s-injective on a compact Riemannian manifold with boundary if the following property holds for all $f\in C^\infty(T_mM)$: $I_Af=0$ if and only if $\widehat{\sigma}_Af=\widehat{\sigma}_A\nabla^A u$ for some $u\in C^\infty(S_{m-1}M)$ vanishing on the boundary.
S-injectivity allows one to decompose the kernel of~$I_A$ as
\begin{equation}
\ker(I_A|_{C^\infty(T_mM)})=\im(\mathcal{H}|_{C^\infty(T_mM)})\oplus\im(\widehat{\sigma}_A\nabla^A|_Y)
\end{equation}
where $\mathcal{H}=\id-\widehat{\sigma}_A$ and
\begin{equation}
Y=\{u\in C^\infty(S_{m-1}M): u|_{\partial M}=0\}.
\end{equation}
It follows that s-injectivity of any mixing ray transform implies s-injectivity for all mixing ray transforms.

\begin{cor}
\label{cor:sinjectivity}
Let $m\geq 1$ and $(M, g)$ be a compact Riemannian manifold with boundary so that the transform~$I_A$ is s-injective for some~$A$ of degree~$m$. Then~$I_{\widetilde{A}}$ is s-injective for all~$\widetilde{A}$ of degree~$m$. 
\end{cor}

\begin{proof}
Let us denote $\widehat{\sigma}_A=A^{-1}\sigma A$ and $\widehat{\sigma}_{\widetilde{A}}=\widetilde{A}^{-1}\sigma \widetilde{A}$ for the projections.
Using the solenoidal injectivity for~$I_A$ we easily obtain
\begin{align}
I_{\widetilde{A}}f=I_A(A^{-1}\widetilde{A}f)=0 &\Leftrightarrow \exists u \in Y:  \widehat{\sigma}_AA^{-1}\widetilde{A}f=\widehat{\sigma}_A\nabla^Au \\ &\Leftrightarrow  \exists u \in Y:\widehat{\sigma}_{\widetilde{A}}f=\widehat{\sigma}_{\widetilde{A}}\nabla^{\widetilde{A}}u.\qedhere
\end{align}
\end{proof}

We immediately obtain the following corollary from the previous corollary.

\begin{cor}
\label{cor:generalsinjectivity}
Take any $m\geq1$.
Assume that $(M, g)$ is a compact Riemannian manifold with boundary so that the geodesic ray transform is s-injective on $m$-tensor fields.
Then~$I_A$ is s-injective for all~$A$ of degree~$m$.
\end{cor}

We have similar results for the mixing ray transform in the quotient space $\X(T_mM)/\ker(\lambda\circ A)$.
We denote by~$[\cdot]_A$ the corresponding equivalence classes and say that the quotient transform defined as $I_A^q[f]_A=I_A f$ (see section~\ref{sec:mixingraydefinition}) is s-injective if for all $[f]_A\in C^\infty(T_mM)/\ker(\lambda\circ A)$ we have $I^q_A [f]_A=0$ if and only if $[f]_A=[\nabla^A u]_A$ for some $u\in C^\infty (S_{m-1}M)$ vanishing on the boundary.

\begin{cor}
Let $m\geq 1$ and~$A$ be a mixing of degree~$m$. Assume that $(M, g)$ is a compact Riemannian manifold with boundary.
Then~$I_A$ is s-injective if and only if~$I^q_A$ is s-injective.
\end{cor}

\begin{proof}
Assume first that~$I_A$ is s-injective.
We obtain
\begin{equation}
I_A f=I^q_A[f]_A=0\Leftrightarrow \widehat{\sigma}_A f=\widehat{\sigma}_A\nabla^A u
\end{equation}
which in turn implies
\begin{equation}
[f]_A=[\widehat{\sigma}_A f]_A=[\widehat{\sigma}_A\nabla^A u]_A=[\nabla^Au]_A.
\end{equation}
Assume then that~$I^q_A$ is s-injective.
Now 
\begin{equation}
I^q_A [f]_A=I_A f=0\Leftrightarrow [f]_A=[\nabla^A u]_A
\end{equation}
which implies $f-\nabla^Au=h\in\ker(\lambda\circ A)$.
Hence $\widehat{\sigma}_A f=\widehat{\sigma}_A\nabla^Au$.
\end{proof}

The previous results imply that if~$I^q_A$ is s-injective for some~$A$ of degree $m\geq 1$, then~$I^q_{\widetilde{A}}$ is s-injective for all~$\widetilde{A}$ of degree~$m$.
We remark that if $A=\id$, then~$I^q_A$ corresponds to the geodesic ray transform in the quotient space $\X(T_mM)/\ker(\lambda)$.
Especially, s-injectivity of~$I$ on $m$-tensor fields implies s-injectivity for~$I_A^q$ where~$A$ is any mixing of degree~$m$.

\subsubsection{Mixed ray transform on compact simple surfaces}
Let us then consider the mixed ray transform $L_{k, l}=I_{SM}\circ\lambda\circ A_{k, l}$ on a compact simple surface $(M, g)$ where~$A_{k, l}$ is defined via equation~\eqref{eq:mixedraytwodimensional}.
Define the operators $\lambda^\prime w=\sigma_{k, l}(g\otimes w)$ and $\der^\prime u=\sigma_l\nabla u$ where~$\sigma_l$ is the symmetrization with respect to the last~$l$ indices and~$\sigma_{k, l}$ is the symmetrization with respect to the first~$k$ and the last~$l$ indices.
In coordinates
\begin{align}
(\lambda^\prime w)_{i_1\dotso i_k j_1\dotso j_l}&=\sigma_{k, l}(g_{i_1 j_1}w_{i_2,\dotso i_k j_2\dotso j_l}) \\
(\der^\prime u)_{i_1\dotso i_kj_1\dotso j_l}&=\sigma_l((\nabla_{e_{j_1}}u)_{i_1\dotso i_kj_2\dotso j_l}).
\end{align}
We compare our approach to the kernel characterization done in~\cite{dHSZ18}.
Especially, we obtain the following alternative result for s-injectivity.

\begin{cor}
\label{cor:mixedraysimplesurface}
Let $(M, g)$ be two-dimensional compact simple Riemannian manifold and $f\in C^\infty(T_{k+l}M)$.
Then $L_{k, l}f=0$ if and only if $\widehat{\sigma}_{A_{k, l}}f=\widehat{\sigma}_{A_{k ,l}}\sigma_l\nabla u$ where $u\in C^\infty(S_kM\otimes S_{l-1}M)$ such that $u|_{\partial M}=0$.
\end{cor}

\begin{proof}
Assume that $L_{k, l}f=0$. Since
\begin{equation}
    \widehat{\sigma}_{A_{k, l}}f\in A^{-1}_{k, l}(C^\infty (S_mM))\subset C^\infty(S_kM\otimes S_lM)
\end{equation}
we obtain $L_{k ,l}\widehat{\sigma}_{A_{k ,l}}f=L_{k, l}f=0$.
By \cite[Theorem 1]{dHSZ18} we have $\widehat{\sigma}_{A_{k ,l}}f=\sigma_l\nabla u+\sigma_{k, l}(g\otimes w)$ for some $u\in C^\infty (S_kM\otimes S_{l-1}M)$, $u|_{\partial M}=0$, and $w\in C^\infty (S_{k-1}M\otimes S_{l-1}M)$.
Now $\sigma_{k, l}(g\otimes w)\in\ker(\lambda\circ A_{k, l})$ and hence $\widehat{\sigma}_{A_{k ,l}}f=\widehat{\sigma}_{A_{k ,l}}\sigma_l\nabla u$. 

Then assume that $\widehat{\sigma}_{A_{k ,l}}f=\widehat{\sigma}_{A_{k ,l}}\sigma_l\nabla u$ for some~$u$ vanishing on the boundary.
Since $L_{k, l}=I_{SM}\circ\lambda\circ A_{k ,l}$ and $\lambda\circ\sigma=\lambda$ we obtain
\begin{equation}
L_{k, l}f=L_{k ,l}\widehat{\sigma}_{A_{k ,l}}f=(I_{SM}\circ\lambda\circ A_{k, l})(A^{-1}_{k, l}\sigma A_{k, l}\sigma_l\nabla u)=L_{k ,l}\sigma_l\nabla u=0
\end{equation}
where the last equality follows from the fundamental theorem of calculus.
\end{proof}

\begin{remark}
The previous s-injectivity result is similar to what we obtained earlier, i.e. $I_{A_{k, l}}f=0$ if and only if $\widehat{\sigma}_{A_{k ,l}}f=\widehat{\sigma}_{A_{k ,l}}\nabla^{A_{k, l}}u$ for some $u\in C^\infty (S_{m-1}M)$ vanishing on the boundary.
We thus have the following alternative characterizations of the kernel of the mixed ray transform
\begin{align}
\ker(L_{k, l}|_{C^\infty(T_mM)})&=\im(\mathcal{H}|_{C^\infty(T_mM)})\oplus \im(\widehat{\sigma}_{A_{k, l}}\nabla^A|_Y) \\
\ker(L_{k, l}|_{C^\infty(T_mM)})&=\im(\mathcal{H}|_{C^\infty(T_mM)})\oplus \im(\widehat{\sigma}_{A_{k, l}}\der^\prime|_{Y^\prime})
\end{align}
where $\mathcal{H}=\id-\widehat{\sigma}_{A_{k, l}}$ and
\begin{align}
Y&=\{u\in C^\infty (S_{m-1}M):u|_{\partial M}=0\}\\
Y^\prime&=\{u\in C^\infty (S_kM\otimes S_{l-1}M): u|_{\partial M}=0\}.
\end{align}
Compare these to the decomposition of the kernel in~\cite{dHSZ18}
\begin{equation}
\label{eq:teemusdecomposition}
\ker(L_{k, l}|_{C^\infty (S_kM\otimes S_lM)})=\im(\lambda^\prime|_{C^\infty (S_{k-1}M\otimes S_{l-1}M)})+\im(\der^\prime|_{Y^\prime}).
\end{equation}
Our decompositions split any tensor field uniquely into the trivial part and non-trivial part of the kernel.
The uniqueness of decomposition~\eqref{eq:teemusdecomposition} is not known and it only applies to tensor fields with certain symmetries.
\end{remark}

\subsubsection{Light ray transform on Lorentzian manifolds}\label{subsec:lightraytransform}
We quickly review the relevant definitions for the light ray transform on static globally hyperbolic Lorentzian manifolds.
More details can be found in~\cite{FIO-light-ray}.

Let $(\mathcal{N},\overline{g})$ be a smooth globally hyperbolic Lorentzian manifold of dimension $1+n$ with signature $(-,+,\dots,+)$.
Let~$\beta$ be a maximal light-like geodesic so that
\begin{equation}
    \overline{\nabla}_{\dot{\beta}(s)}\dot{\beta}(s) = 0,\quad \overline{g}(\dot{\beta}(s),\dot{\beta}(s))=0
\end{equation}
where~$\overline{\nabla}$ is the covariant derivative with respect to~$\overline{g}$.
We define the light ray transform of $f \in C_c^\infty(T_m\mathcal{N})$ as
\begin{equation}\label{eq:lighraytransform}
    \mathcal{L}_\beta f = \int_{-\infty}^{\infty}(\lambda f)(\beta(s),\dot{\beta}(s))\der s,
\end{equation}
where~$\beta$ ranges over all lightlike geodesics of~$\mathcal{N}$.
Since $(\mathcal{N}, \overline{g})$ is globally hyperbolic, there exists a Cauchy hypersurface $N\subset \mathcal{N}$, i.e. a hypersurface such that any causal curve intersects~$N$ exactly once.
We define $g = \overline{g}|_N$; note that $(N, g)$ becomes a Riemannian manifold.
We will focus on static Lorentzian manifolds.
It follows that if~$\mathcal{N}$ is static, then for any Cauchy hypersurface $N \subset \mathcal{N}$ there exists an isometric embedding $\Phi\colon \R \times N \to \mathcal{N}$ so that $\Phi^* \overline{g} = -\kappa \der t^2 + g$ where~$\kappa$ is a smooth positive function on~$N$. We let $g_c=\kappa^{-1}g$.

Let~$r$ be the restriction to the set
\begin{equation}
\Omega=\{(x, v)\in T\mathcal{M}: \overline{g}_x(v, v)=0\}
\end{equation}
where $\mathcal{M}=\Phi(\R\times M)$, $M\subset N$ is a compact submanifold with smooth boundary and $\lambda_r=r\circ\lambda$ as before.
We define the quotient light ray transform~$\mathcal{L}_\beta^q$ in $C^\infty_c(T_m\mathcal{M})/\ker(\lambda_r)$ as $\mathcal{L}_\beta^q[f]=\mathcal{L}_\beta f$; note that the definition does not depend on the representative. We obtain the following s-injectivity result for~$\mathcal{L}_{\beta}^q$.

\begin{cor}
\label{cor:lorentzinjectivity}
Let $(\mathcal{N}, \overline{g})$ be static globally hyperbolic Lorentzian manifold of dimension $1+n$ and let $N\subset\mathcal{N}$ be a fixed Cauchy hypersurface.
Let $\mathcal{M}=\Phi(\R\times M)$ where $M\subset N$ is a compact $n$-dimensional submanifold with smooth boundary and~$\Phi$ is the isometric embedding introduced earlier.
Assume that the geodesic ray transform is s-injective on $(M, g_c)$ and let $[f]\in C_c^{\infty}(T_m\mathcal{M})/\ker(\lambda_r)$. Then $\mathcal{L}_{\beta}^q[f]=0$ for all maximal~$\beta$ in $(\mathcal{M}, \overline{g})$ if and only if $[f]=[\overline{\nabla}T]$ for some $T\in C_c^{\infty}(S_{m-1}\mathcal{M})$.
\end{cor}

\begin{proof}
Assume that $\mathcal{L}_{\beta}^q[f]=0$. Then $\mathcal{L}_\beta(\sigma f)=\mathcal{L}_\beta f=0$ and by \cite[Theorem 2]{FIO-light-ray} we obtain $\sigma f=\sigma\overline{\nabla}T+\sigma(g\otimes U)$ for some $T\in C^{\infty}_c(S_{m-1}\mathcal{M})$ and $U\in C^{\infty}_c(S_{m-2}\mathcal{M})$. Hence 
\begin{equation}
[f]=[\sigma f]=[\sigma\overline{\nabla}T]+[\sigma (g\otimes U)]=[\overline{\nabla}T]
\end{equation}
where we used the fact that $\ker(\lambda)\subset\ker(\lambda_r)$ and $\sigma (g\otimes U)\in\ker(\lambda_r)$. This gives the other direction of the claim. Assume then that $[f]=[\overline{\nabla}T]$ for some $T\in C^\infty_c(S_{m-1}\mathcal{M})$. The fundamental theorem of calculus implies that $\mathcal{L}_\beta^q[f]=\mathcal{L}_\beta^q[\overline{\nabla}T]=\mathcal{L}_\beta(\sigma\overline{\nabla}T)=0$. This concludes the proof. \qedhere

\end{proof}

\begin{remark}
One can realize the quotient space $\X(T_m\mathcal{M})/\ker(\lambda_r)$ as a complementary subspace $V_r\subset\X(T_m\mathcal{M})$ which satisfies $\ker(\lambda_r)\oplus V_r=\X(T_m\mathcal{M})$. This can be done for example by taking the orthogonal complement $V_r=\ker(\lambda_r)^\perp$ with respect to a Riemannian metric on~$\mathcal{M}$ (see section~\ref{sec:decomp}). Then corollary~\ref{cor:lorentzinjectivity} implies that we have the decomposition
\begin{equation}
\ker(\mathcal{L}_{\beta}|_{C_c^{\infty}(T_m\mathcal{M})})=\im(\mathcal{H}|_{C^\infty_c(T_m\mathcal{M})})\oplus\im(\widehat{\sigma}_r\overline{\nabla}|_{C_c^{\infty}(S_{m-1}\mathcal{M})})
\end{equation}
where~$\widehat{\sigma}_r$ is the orthogonal projection onto~$\ker(\lambda_r)^\perp$ and $\mathcal{H}=\id-\widehat{\sigma}_r$.
\end{remark}

\subsection{Boundedness and pointwise estimates of mixings}

In this section we give sufficient conditions which imply pointwise norm estimates and continuity of~$A$ in Sobolev spaces.
Boundedness and pointwise estimates are used in section~\ref{sec:mainthms} to prove stability estimates and injectivity results for the mixed ray transform on two-dimensional orientable Riemannian manifolds.

\begin{lem}
\label{lemma:boundednessofmixedray}
Let $f\in C^q(T_mM)$ where $m\geq 1$ and $q\in\mathbb{N}$.
Then the following properties hold:
\smallskip
\begin{enumerate}[(a)]
    \item\label{item:boundedness1}
    If~$A_i$ satisfy the relation $\abs{A_i v}_{g_x}\leq C_i(x)\abs{v}_{g_x}$ for all $v\in T_x M$, then we have the pointwise estimate 
    \begin{equation}
    \abs{Af}_{g_x}\leq n^mC_1(x)\dotso C_m(x)\abs{f}_{g_x}.
    \end{equation}
    Especially, if $C_i=C_i(x)$ are all bounded, then~$A$ extends into a bounded mapping $A\colon L^2(T_mM)\rightarrow L^2(T_mM)$.
    \smallskip
    \item\label{item:sobolevboundedness}
    If in addition $\abs{\nabla_{e_j}A_i}_{g_x}\leq C^{\prime}_i(x)$ for any local frame $\{e_j\}$, then we have the pointwise estimate
    \begin{equation}
    \abs{\nabla(Af)}_{g_x}\leq C''(x)(\abs{f}_{g_x}+\abs{\nabla f}_{g_x})
    \end{equation}
    where $C''=C''(x)$ can be expressed in terms of~$C_i$ and~$C_i^{\prime}$.
    Especially, if $C_i, C^{\prime}_i$ are all bounded, then~$A$ extends into a bounded mapping $A\colon H^1(T_m M)\rightarrow H^1(T_mM)$.
    \smallskip
    \item\label{item:mixedisometry}
    If $(M, g)$ is a two-dimensional orientable Riemannian manifold, then the operator~$A_{k, l}$ defined in~\eqref{eq:mixedraytwodimensional}  satisfies
    \begin{align}
    \abs{\nabla^p(A_{k, l}f)}_{g_x}&=\abs{\nabla^p f}_{g_x}
    \end{align}
    for all $p \in \Nset$, $p\leq q$.
    In particular, the mixing~$A_{k, l}$ extends into an isometry $A_{k, l}\colon H^p(T_mM)\rightarrow H^p(T_mM)$ for all $p\in\Nset$.
    
\end{enumerate}
\end{lem}

\begin{proof}
(\ref{item:boundedness1}) Choose normal coordinates in a neighborhood of~$x$.
The boundedness assumption for~$A_i$ implies 
\begin{equation}
\abs{(A_i)^j_k(x)}\leq \bigg(\sum_{j=1}^n\abs{(A_i)^j_k(x)}^2\bigg)^{1/2}=\abs{A_ie_k}_{g_x}\leq C_i(x)
\end{equation}
where~$(A_i)^j_k$ are the components of~$A_i$ in these coordinates.
Now we can estimate the norm as
\begin{equation}
\begin{split}
\abs{Af}^2_{g_x}
&=
\sum_{i_1\dotso i_m=1}^n((Af)_{i_1\dotso i_m}(x))^2
\\&=
\sum_{i_1\dotso i_m=1}^n\bigg(\sum_{j_1\dotso j_m=1}^n(A_1)^{j_1}_{i_1}(x)\dotso (A_m)^{j_m}_{i_m}(x)f_{j_1\dotso j_m}(x)\bigg)^2 
\\&\leq
n^mC_1^2(x)\dotso C_m^2(x)\sum_{i_1\dotso i_m=1}^n\sum_{j_1\dotso j_m=1}^n\abs{f_{j_1\dotso j_m}(x)}^2
\\&\leq
n^{2m}C_1^2(x)\dotso C_m^2(x)\abs{f}_{g_x}^2.
\end{split}
\end{equation}
If~$C_i$ are all bounded, then $A\colon L^2(T_mM)\rightarrow L^2(T_mM)$ is bounded by definition and approximation by smooth tensor fields.

(\ref{item:sobolevboundedness}) By choosing normal coordinates, the covariant derivative at the point $x\in M$ reduces to the ordinary derivative.
Now
\begin{equation}
\abs{\nabla_{e_j} A_i}_{g_x}^2=\sum_{k, l=1}^n(\partial_j (A_i)^k_l(x))^2,
\end{equation}
which implies $\abs{\partial_j(A_i)^k_l(x)}\leq C^{\prime}_i(x)$.
Using the Leibniz rule we obtain
\begin{equation}
\begin{split}
\abs{\nabla (Af)}_{g_x}^2
&=
\sum_{k, i_1\dotso i_m=1}^n (\partial_k(Af)_{i_1\dotso i_m}(x))^2
\\&\leq
n^{2m+1}((C^{\prime}_1)^2(x)\dotso C^2_m(x)+\dotso +C_1^2(x)\dotso (C_m^{\prime})^2(x))\abs{f}_{g_x}^2
\\&\quad+
n^{2m}C_1^2(x)\dotso C_m^2(x)\abs{\nabla f}_{g_x}^2
\\&=
\widehat{C}(x)\abs{f}_{g_x}^2+\widetilde{C}(x)\abs{\nabla f}_{g_x}^2.
\end{split}
\end{equation}
By taking $C''(x)=\sqrt{2}\max\{\widehat{C}^{1/2}(x), \widetilde{C}^{1/2}(x)\}$ we get the desired inequality. If $C_i, C^{\prime}_i$ are all bounded, then $A\colon H^1(T_mM)\rightarrow H^1(T_mM)$ is bounded.

(\ref{item:mixedisometry}) Again using normal coordinates, one can calculate that
\begin{equation}
    g_x(A_{k, l}f, A_{k, l}f)=g_x(f, A^{-1}_{k, l}A_{k, l}f)=g_x(f, f),
\end{equation}
where we used the relations $(A_i)^j_m=-(A_i)^m_j$ for $i=1, \dotso k$, $(A_i)^j_m=\delta^j_m$ for $i=k+1, \dotso k+l$ and $(-1)^kA_{k, l}=A^{-1}_{k, l}$.
For the derivatives we get
\begin{equation}
    g_x(\nabla^p(Af), \nabla^p(Af))=g_x(\nabla^pf, \nabla^pf)
\end{equation}
using the fact that $\sum_{j=1}^n(A_i)^m_j(A_i)^q_j=\delta^m_q$ for all $i=1, \dotso k+l$. \qedhere
\end{proof}

\begin{remark}
In a similar fashion as in part~(\ref{item:sobolevboundedness}) one obtains the boundedness of $A\colon H^k(T_mM)\rightarrow H^k(T_mM)$ if one assumes boundedness of the derivatives up to order $k\in\mathbb{N}$, i.e. $\abs{\nabla_{\alpha}A_i}_{g_x}\leq C^{\alpha}_i(x)$ for all $\abs{\alpha}\leq k$ where $C^{\alpha}_i=C^{\alpha}_i(x)$ is bounded, $\nabla_{\alpha}=\nabla^{\alpha_1}_{e_1}\dotso\nabla^{\alpha_n}_{e_n}$ and $\abs{\alpha}=\alpha_1+\dotso +\alpha_n$.
\end{remark}

\section{The mixed and transverse ray transforms on two-dimensional orientable Riemannian manifolds}
\label{sec:mainthms}

\subsection{Solenoidal injectivity on compact and non-compact surfaces}

First we state a general result on s-injectivity of the mixed ray transform on compact orientable surfaces with boundary.
This follows from corollary~\ref{cor:generalsinjectivity}.
We use the notation introduced in section~\ref{subsubsection:generalsinjectivity}.
Note that $A_{k, l}f\in C^{p}(T_mM)$ whenever $f\in C^{p}(T_mM)$ where $p\in\mathbb{N}$.

\begin{cor}
\label{cor:sinjectivitymixed}
Let $m\geq 1$. Let $(M, g)$ be compact two-dimensional orientable Riemannian manifold with boundary such that the geodesic ray transform is s-injective on $C^{\infty}(S_mM)$ and let $f\in C^{\infty}(T_mM)$.
Then $L_{k, l}f=0$ if and only if $\widehat{\sigma}_{A_{k, l}}f=\widehat{\sigma}_{A_{k, l}}\nabla^{A_{k, l}} h$ for some $h\in C^{\infty}(S_{m-1}M)$ vanishing on the boundary~$\partial M$.
\end{cor}

We note that the previous result holds on a wide class of two-dimensional orientable manifolds.
These include for example compact simple surfaces~\cite{PSU-tensor-tomography-on-simple-surfaces} and simply connected compact surfaces with strictly convex boundary and non-positive sectional curvature \cite{PS-sharp-stability-nonpositive-curvature, SHA-integral-geometry-tensor-fields}.
See \cite{IM:integral-geometry-review, PSU-tensor-tomography-progress} for more manifolds with s-injective geodesic ray transform.

We have the following corollary for the mixed ray transform on Cartan--Hadamard manifolds which is a simple consequence of the pointwise estimates for~$A_{k, l}$ and the results in~\cite{LRS-tensor-tomography-cartan-hadamard}.
We denote by~$K(x)$ the Gaussian curvature of $(M, g)$ at $x\in M$.

\begin{cor}
\label{cor:cartanhadamardinjectivity}
Let $(M, g)$ be a two-dimensional Cartan--Hadamard manifold and let $m\geq 1$.
The following claims are true:
\smallskip
\begin{enumerate}[(a)]
    \item\label{item:cartan1}
    Let $-K_0\leq K\leq 0$ for some $K_0>0$ and $f\in E^1_{\eta}(T_mM)$ for some $\eta>\frac{3}{2}\sqrt{K_0}$.
    Then $L_{k, l}f=0$ if and only if $ \widehat{\sigma}_{A_{k, l}}f=\widehat{\sigma}_{A_{k, l}}\nabla^{A_{k, l}} h$ for some $h\in S_{m-1}M$ such that $h\in E_{\eta-\epsilon}(T_{m-1}M)$ for all $\epsilon>0$. 
    \medskip
    \item\label{item:cartan2}
    Let $K\in P_{\kappa}(M)$ for some $\kappa>2$ and $f\in P^1_{\eta}(T_mM)$ for some $\eta>2$.
    Then $L_{k, l}f=0$ if and only if $ \widehat{\sigma}_{A_{k, l}}f=\widehat{\sigma}_{A_{k, l}}\nabla^{A_{k, l}} h$ for some $h\in S_{m-1}M\cap P_{\eta-1}(T_{m-1}M)$.
\end{enumerate}
\end{cor}

\begin{proof}
(\ref{item:cartan1}) If $f\in E^1_{\eta}(T_mM)$, then from the pointwise estimates for the transform~$A_{k, l}$ we obtain that 
\begin{equation}
\abs{\sigma A_{k, l} f}_{g_x}\leq (m!)^{1/2}\abs{A_{k, l}f}_{g_x}= (m!)^{1/2}\abs{f}_{g_x}\leq C
e^{-\eta d(x, o)}
\end{equation}
for some $C>0$ and
\begin{equation}
\abs{\nabla(\sigma A_{k, l} f)}_{g_x}\leq (m!)^{1/2}\abs{\nabla(A_{k, l} f)}_{g_x}= (m!)^{1/2}\abs{\nabla f}_{g_x}\leq C^{\prime}e^{-\eta d(x, o)}
\end{equation}
for some $C^{\prime}>0$.
Hence $\sigma A_{k, l} f\in S_mM\cap E^1_{\eta}(T_mM)$ for some $\eta>\frac{3}{2}\sqrt{K_0}$.
Since $I(\sigma A_{k, l}f)=L_{k, l}f=0$, we must have $\sigma A_{k, l} f=\sigma\nabla h$ for some $h\in S_{m-1}M$ where $h\in E_{\eta-\epsilon}(T_{m-1}M)$ for all $\epsilon>0$ by \cite[Theorem 1.1]{LRS-tensor-tomography-cartan-hadamard}.
This gives the claim for the first part.

(\ref{item:cartan2}) Similarly using the pointwise estimates one obtains that $\sigma A_{k, l} f\in S_mM\cap P^1_{\eta}(T_mM)$ for some $\eta>2$.
Now \cite[Theorem 1.2]{LRS-tensor-tomography-cartan-hadamard} implies that $\sigma A_{k, l} f=\sigma\nabla h$ for some $h\in S_{m-1}M\cap P_{\eta-1}(T_{m-1}M)$.
This proves the second part.
\end{proof}

\begin{remark}
One can study the mixed ray transform on asymptotically hyperbolic surfaces~\cite{GGSU-asymptotically-hyperbolic-manifolds}.
Let $(M,g)$ be an asymptotically hyperbolic surface, $\ol{M}$ the compactification of~$M$ and~$\rho$ a geodesic boundary defining function as defined in~\cite{GGSU-asymptotically-hyperbolic-manifolds}.
One usually assumes that $f \in \rho^{1-m} C^\infty(S_m\ol{M})$ to obtain s-injectivity results for the geodesic ray transform.
It then follows that $\sigma A_{k, l}f\in\rho^{1-m} C^{\infty}(S_m\overline{M})$ and similar s-injectivity result as in corollary~\ref{cor:sinjectivitymixed} holds under certain assumptions on $(M, g)$; we refer to~\cite{GGSU-asymptotically-hyperbolic-manifolds} for a more detailed discussion.
One can also study the mixed ray transform on asymptotically conic surfaces $(M^{\prime}, g^{\prime})$.
One obtains s-injectivity for tensor fields $f\in A_{k, l}^{-1}\rho^{\prime r}C^{\infty}(S_m^{sc}\overline{M^{\prime}})$ where $\rho^\prime$ is the boundary defining function, $r>n/2+1$ and $S_m^{sc}\overline{M^\prime}\subset S_m\overline{M^\prime}$ is the set of scattering tensor fields on the compactification~$\overline{M^\prime}$.
See~\cite{GLT-asymptotically-conic-spaces} for more details.
\end{remark}

\subsection{Stability results on compact surfaces}\label{sec:stability}

In this section we obtain stability estimates for the mixed ray transform.
We begin with the following lemma.

\begin{lem}\label{lemma:normaloperator}
Let $(M, g)$ be a compact simple surface.
Then the normal operator of the mixed ray transform~$L_{k,l}$ is $N_{k, l}=(-1)^kA_{k, l}N A_{k, l}$ where~$N$ is the normal operator of the geodesic ray transform~$I$ on $(k+l)$-tensor fields.
\end{lem}

\begin{proof}
By theorem~\ref{thm:linalg} part~(\ref{item:adjoint}) we only need to calculate $(\mathcal{D}^{-1})^*=A_{k, l}^*$.
Now for the matrix representations of~$A_i$ we have that $(A_i)^j_m=-(A_i)^m_j$ for $i=1, \dotso k$ and $(A_i)^j_m=\delta^j_m$ for $i=k+1, \dotso k+l$.
Using this one obtains
\begin{equation}
g_x(A_{k, l}f, h)
=
g_x(f, (-1)^k A_{k, l}h)
\end{equation}
and thus
\begin{equation}
\sis{A_{k, l}f, h}_{L^2(T_mM)}
=
\sis{f, (-1)^k A_{k, l}h}_{L^2(T_mM)}.
\end{equation}
Hence $A_{k, l}^*=(-1)^kA_{k, l}$ which gives the claim.
\end{proof}

The next estimates are direct consequences of the results in \cite{PS-sharp-stability-nonpositive-curvature,STE-sharp-stability-estimate-2-tensors,  SU-stability-tensor-fields-boundary-rigidity}.
We denote by~$\Sol(T_mM)$ the set of solenoidal tensor fields.
For the definition of the tangential norm $\norm{\cdot}_{H_T^{1/2}(\partial_{\mathrm{in}} SM)}$ see~\cite{PS-sharp-stability-nonpositive-curvature}. 

\begin{cor}
\label{cor:stability}
For any compact simple surface $(M,g)$ and nonnegative integers~$k$ and~$l$ there is a constant $C > 0$ so that:
\smallskip
\begin{enumerate}[(a)]
    \item\label{item:simple1}
    Let $k+l=1$.
    Let~$g$ be extended to a simple metric in $M_1\supset\supset M$.
    Then the estimate
    \begin{align}
    \norm{f}_{L^2(T_1M)}/C&\leq\norm{N_{k, l}f}_{H^1(T_1M_1)}\leq C\norm{f}_{L^2(T_1M)}
    \end{align}
    holds for all $f\in A_{k, l}^{-1}(\Sol(T_1M))\cap L^2(T_1M)$.
    \medskip
    \item\label{item:simple2}
    Let $k+l=2$.
    Let~$g$ be extended to a simple metric in $M_1\supset\supset M$.
    Then the estimate
    \begin{align}
    \norm{f}_{L^2(T_2M)}/C&\leq\norm{N_{k, l}f}_{H^1(T_2M_1)}\leq C\norm{f}_{L^2(T_2M)}
    \end{align}
    holds for all $f\in A_{k, l}^{-1}(\Sol(T_2M)\cap H^1(S_2M))$.
    \medskip
    \item\label{item:simple3}\label{item:sharpestimate}
    Let $m\coloneqq k+l\geq 1$.
    Assume further that $(M, g)$ has non-positive sectional curvature.
    Then the estimate 
    \begin{equation}
    \norm{f}_{L^2(T_mM)}\leq C\norm{L_{k, l}f}_{H^{1/2}_T(\partial_{\mathrm{in}}SM)}
    \end{equation}
    holds for all $f\in A_{k, l}^{-1}(\Sol(T_mM)\cap H^1(S_mM))$.
\end{enumerate}
\end{cor}

\begin{proof}
(\ref{item:simple1}) We know that the stability estimate holds for the geodesic ray transform \cite[Theorem 4]{SU-stability-tensor-fields-boundary-rigidity}.
Now $A_{k, l}\colon L^2(T_1M)\rightarrow L^2(T_1M)$ and $A_{k, l}^*=A^{-1}_{k, l}\colon H^1(T_1M_1)\rightarrow H^1(T_1 M_1)$ are isometries by lemma~\ref{lemma:boundednessofmixedray} part~(\ref{item:mixedisometry}).
By theorem~\ref{thm:linalg} part~(\ref{item:stabilitynormaloperators}) we obtain
\begin{align}
\norm{f}_{L^2(T_1M)}/C&\leq\norm{N_{k, l}f}_{H^1(T_1M_1)}\leq C\norm{f}_{L^2(T_1M)}.
\end{align}

(\ref{item:simple2})
By \cite[Theorem 1]{STE-sharp-stability-estimate-2-tensors} the stability estimate holds for the geodesic ray transform if we know s-injectivity.
But s-injectivity holds on two-dimensional simple manifolds for tensor fields of all order \cite[Theorem 1.1]{PSU-tensor-tomography-on-simple-surfaces}.
Using the fact that $A_{k, l}\colon L^2(T_2M)\rightarrow L^2(T_2M)$ and $A_{k, l}^*\colon H^1(T_2M)\rightarrow H^1(T_2M)$ are isometries we obtain the stability estimate as in part~(\ref{item:simple1}) above.

(\ref{item:simple3})
We know that the stability estimate is true for the geodesic ray transform \cite[Theorem 1.3]{PS-sharp-stability-nonpositive-curvature}.
Since $A_{k, l}\colon L^2(T_mM)\rightarrow L^2(T_mM)$ is an isometry theorem~\ref{thm:linalg} part~(\ref{item:stabilitygeodesictransform}) implies
\begin{equation}
\norm{f}_{L^2(T_mM)}\leq C\norm{L_{k, l} f}_{H^{1/2}_T(\partial_{\mathrm{in}} SM)}.
\end{equation}
This concludes the proof.
\end{proof}

\begin{remark}
Note that for example the estimate 
\begin{equation}
\norm{f}_{L^2}/C\leq\norm{N_{k, l}f}_{H^1}\leq C\norm{f}_{L^2}
\end{equation}
holds for all $f\in A^{-1}_{k, l}(S^{\prime\prime})$ if and only if the estimate 
\begin{equation}
\norm{h}_{L^2}/C\leq\norm{Nh}_{H^1}\leq C\norm{h}_{L^2}
\end{equation}
holds for all $h\in S^{\prime\prime}$.
This follows since $A_{k, l}\colon H^p(T_mM)\rightarrow H^p(T_mM)$ is an isometry for all $p\in\mathbb{N}$ and $N_{k, l}=(-1)^kA_{k, l}NA_{k, l}$.
Therefore the sets defined in corollary~\ref{cor:stability} are in a sense largest sets where such stability estimates can hold.
A similar sharp stability estimate as in part~(\ref{item:sharpestimate}) of corollary~\ref{cor:stability} can be proved on compact simple surfaces when $m=1, 2$ \cite[Theorem 1.1]{AS-sharp-estimates-for-simple-manifolds} (see also~\cite{BS-stability-estimates-in-tensor-tomography} for the Euclidean case).
\end{remark}

\subsection{Transverse ray transform of one-forms}
Next we study the kernel of the transverse ray transform on one-forms in two dimensions. The result which we obtain is previously known in~$\R^2$~\cite{DS-tomography, NA-mathematical-methods-image-reconstruction}.
We recall that in our notation the transverse ray transform is $I_{\perp}f=I_Af$ where $A_i=\star$ for all $i\in\{1,\dotso, m\}$. For a scalar field~$\phi$, we define $\curl(\phi)=e_2(\phi)e^1-e_1(\phi)e^2$ where $\{e_1, e_2\}$ is any positively oriented local orthonormal frame and $\{e^1, e^2\}$ its coframe.

\begin{cor}
\label{cor:transverseinjectivity}
Let $(M, g)$ be two-dimensional orientable Riemannian manifold with boundary such that the geodesic ray transform is s-injective on smooth one-forms and let $f\in C^\infty(T_1M)$.
Then $I_{\perp}f=0$ if and only if $f=\curl(\phi)$ for some smooth function~$\phi$ vanishing on the boundary. 
\end{cor}

\begin{proof}
If $f=\curl(\phi)$ where~$\phi$ vanishes on the boundary, then $Af=\der\phi$ and $I_{\perp}f=I(Af)=0$ by the fundamental theorem of calculus.
For the converse, if $I_{\perp}f=0$, then $I(Af)=0$.
By solenoidal injectivity we have that $Af=\der\phi$ for some smooth scalar function~$\phi$ vanishing on the boundary~$\partial M$.
This implies that $f=A^{-1}\der\phi$ which in local positively oriented orthonormal frame $\{e_1, e_2\}$ means $f_1=e_2(\phi)$ and $f_2=-e_1(\phi)$, i.e. $f=\curl(\phi)$.
\end{proof}

\begin{remark}
We note that on two-dimensional Cartan--Hadamard manifolds one can also deduce from $I_\perp f=0$ that $f=\curl(\phi)$ if one of the following assumptions holds
\begin{enumerate}[(a)]
    \item $-K_0\leq K\leq 0$ for some $K_0>0$ and $f\in E^1_{\eta}(T_1M)$ for some $\eta>\frac{3}{2}\sqrt{K_0}$
    \smallskip
    \item $K\in P_{\kappa}(M)$ for some $\kappa>2$ and $f\in P^1_{\eta}(T_1M)$ for some $\eta>2$.
\end{enumerate}
\end{remark}

If we combine the data from the geodesic ray transform~$If$ and the transverse ray transform~$I_{\perp}f$, we can uniquely reconstruct any smooth one-form on two-dimensional compact simple manifolds.
This result is also known previously in~$\mathbb{R}^2$~\cite{BH-tomographic-reconstruction-vector-fields}.
Recall that $\Delta_g u=\text{div}((\text{grad}(u))$ where $\text{grad}(u)=(\der u)^{\sharp}$.

\begin{cor}
\label{cor:simplesurfaceuniqueness}
Let $(M, g)$ be a compact simple surface.
Then the geodesic ray transform and the transverse ray transform together determine $f\in C^\infty(T_1M)$ uniquely, i.e. if both $If=0$ and $I_{\perp}f=0$, then $f=0$.
\end{cor}

\begin{proof}
Since $(M, g)$ is simple, the solenoidal injectivity of~$I$ (see~\cite{PSU-tensor-tomography-on-simple-surfaces}) implies that $f=\der u$ for some smooth function~$u$ vanishing on the boundary.
On the other hand, $I_{\perp}f=0$ gives that $f=\curl(\phi)$ for some smooth scalar function~$\phi$ by corollary~\ref{cor:transverseinjectivity}.
But this implies that $\text{div}(f)=0$.
Therefore $\Delta_g u=\text{div}(f)=0$ so~$u$ is a harmonic function vanishing on the boundary.
We obtain $u=0$ and hence $f=0$.
\end{proof}

\begin{remark}
One could also use solenoidal decomposition to prove the previous corollary.
By the solenoidal decomposition $f=\sol{f}+\der u$.
Now $If=0$ implies that $\sol{f}=0$.
On the other hand, $I_{\perp}f=0$ implies that~$f$ is solenoidal, i.e. $f=\sol{f}=0$.
\end{remark}

The previous corollary holds also on two-dimensional Cartan--Hadamard manifolds as we will prove next.
We first state and prove a version of Liouville's theorem on Cartan--Hadamard manifolds.

\begin{lem}
\label{lemma:cartanharmonic}
Let $(M, g)$ be a two-dimensional Cartan--Hadamard manifold and~$u$ harmonic function on~$M$, i.e. $\Delta_g u=0$.
Fix any point $o\in M$.
Assume that one of the following conditions hold:
\begin{enumerate}[(a)]
\item\label{item:harmonic1} $-K_0\leq K\leq 0$ for some $K_0>0$ and
\begin{equation}
    \abs{u(x)}\abs{\der u(x)}
    \leq
    Ce^{-\eta d(x,o)}
\end{equation}
for some $C>0$ and some $\eta>\sqrt{K_0}$.
\smallskip
\item\label{item:charmonic2} The curvature satisfies
\begin{equation}
\label{eq:K-in-P}
    \abs{K(x)}
    \leq
    C(1+d(x,o))^{-\kappa}
\end{equation}
for some $C>0$ and $\kappa>2$ and the function satisfies
\begin{equation}
    \abs{u(x)}\abs{\der u(x)}
    \leq
    C(1+d(x,o))^{-\eta}
\end{equation}
for some $\eta>1$.
\end{enumerate}
Then~$u$ is constant.
\end{lem}

We point out that the conditions above are independent of the choice of $o\in M$ as in the definition of the spaces in~\eqref{eq:cartanhadamardspaces}.
Moving the point will only change the constants.

\begin{proof}
Assume first that~(\ref{item:harmonic1}) holds.
Let~$B_r(o)$ be the geodesic ball of radius $r>0$ centered at~$o$.
Using the integration by parts formula (see~\cite{LEE-riemannian-manifolds}) we obtain
\begin{equation}
\begin{split}
0
&=
\int_M u\Delta u\der V_g
\\&=
\lim_{r\rightarrow\infty}\int_{B_r(o)}u\Delta u \der V_g
\\&=
\lim_{r\rightarrow\infty}\bigg(-\int_{B_r(o)}\abs{\text{grad}(u)}_{g_x}^2\der V_g+\int_{S_r(o)}u N(u)\der\widehat{V}_g\bigg)
\end{split}
\end{equation}
where $\der\widehat{V}_g$ is the induced volume form on the geodesic sphere $S_r(o)=\partial B_r(o)$ and~$N$ is the outward unit normal vector field.
We focus on the second term.
Since $N(u)=g_x(\text{grad}(u), N)$ and $\abs{\text{grad}(u)}_{g_x}=\abs{\der u}_{g_x}$, we can estimate that $\abs{u N(u)}\leq\abs{u} \abs{\text{grad}(u)}_{g_x}=\abs{u}\abs{\der u}_{g_x}$.
The volume form can be expressed in polar coordinates as $\der\widehat{V}_g=J_o(r, \theta)\der\theta$ where $\abs{J_o(r, \theta)}\leq Ce^{\sqrt{K_0}r}$ \cite[Lemma 4.7]{LE-cartan-hadamard}.
Therefore we obtain
\begin{equation}
\abs{\int_{S_r(o)}u N(u)\der\widehat{V}_g}\leq C^{\prime}e^{(-\eta+\sqrt{K_0})r}\xrightarrow{r\rightarrow\infty} 0.
\end{equation}
This implies $\abs{\der u}_{g_x}=\abs{\text{grad}(u)}_{g_x}=0$ and hence $\der u=0$.
Connectedness of~$M$ implies that~$u$ is constant.

If~(\ref{item:charmonic2}) holds, then $\abs{J_o(r, \theta)}\leq Cr$ \cite[Lemma 4.7]{LE-cartan-hadamard}. The claim is proved identically as in part~(\ref{item:harmonic1}).
\end{proof}

\begin{remark}
One can prove the previous lemma in the exact same way for Cartan--Hadamard manifolds of dimension $n>2$ using the growth estimates for the Jacobi fields proved in~\cite{LRS-tensor-tomography-cartan-hadamard}.
In the condition~(\ref{item:harmonic1}) one requires $\eta>(n-1)\sqrt{K_0}$ and in the condition~(\ref{item:charmonic2}) one requires $\eta>n-1$.
\end{remark}

\begin{cor}
\label{cor:cartanhadamarduniqueness}
Let $(M, g)$ be two-dimensional Cartan--Hadamard manifold.
Assume that one of the following conditions holds:
\begin{enumerate}[(a)]
    \item\label{item:choneform1} $-K_0\leq K\leq 0$ for some $K_0>0$ and $f\in E^1_{\eta}(T_1M)$ for some $\eta>\frac{3}{2}\sqrt{K_0}$.
    \smallskip
    \item\label{item:choneform2} The curvature satisfies the estimate~\eqref{eq:K-in-P} for some $C>0$ and $\kappa>2$ and $f\in P^1_{\eta}(T_1M)$ for some $\eta>2$.
\end{enumerate}
Then the geodesic ray transform and the transverse ray transform together determine the one-form~$f$ uniquely, i.e. if both $If=0$ and $I_{\perp}f=0$, then $f=0$.
\end{cor}

\begin{proof}
Assume that~(\ref{item:choneform1}) holds.
The condition $If=0$ implies that $f=\der h$ for some $h\in E_{\eta-\epsilon}(M)$ where $\epsilon>0$ is arbitrary (see~\cite{LRS-tensor-tomography-cartan-hadamard}).
On the other hand, $I_{\perp}f=0$ implies that $\Delta_g h=\text{div}(f)=0$.
Hence~$h$ is harmonic and satisfies the decay estimate in lemma~\ref{lemma:cartanharmonic}.
Thus~$h$ is constant and $f=0$.
The proof under the assumption~(\ref{item:choneform2}) is identical.
\end{proof}

\appendix

\section{Notation}
\label{appendix}

\subsection{Integral transforms}
\begin{itemize}
\item $If$, the geodesic X-ray transform of a tensor field~$f$ of order~$m$. See section~\ref{subsec:geodesicraytransform} and equations~\eqref{eq:grtbdry} and~\eqref{eq:geodesicraycartanhadamard}.
\item $I_{SM}h$, the geodesic ray transform of a function $h\colon SM\rightarrow\mathbb{R}$. See section~\ref{subsec:geodesicraytransform} and equation~\eqref{eq:geodesicrayspherebundle}.
\item $I_{A, r}f$, the (abstract) mixing ray transform with a mixing~$A$ of degree~$m$, operating on a tensor field~$f$ of order~$m$. See section~\ref{sec:mixingraydefinition} and equation~\eqref{eq:abstractmixingraytransform}.
\item $L_{k,l}f=I_{A_{k,l}}f$, the mixed ray transform of a tensor field~$f$ of order $k+l$ on two-dimensional orientable Riemannian manifold. See section~\ref{sec:definitionofmixedray} and equations~\ref{eq:mixedraytwodimensional} and~\ref{eq:mixedraydefinition}.
\item $I_\perp f$, the transverse ray transform of a tensor field~$f$ of order $k$, corresponding to the mixed ray transform with $l=0$. See section~\ref{sec:definitionofmixedray} and equation~\ref{eq:mixedraytwodimensional}.
\item $I^q_{A, r}[f]=I_{A, r}f$, the quotient transform of an equivalence class of tensor field~$f$ of degree~$m$. See section~\ref{sec:mixingraydefinition}.
\item $\mathcal{L}_\beta f$, the light ray transform of a (compactly supported) tensor field of order~$m$. See section~\ref{subsec:lightraytransform} and equation~\eqref{eq:lighraytransform}.
\item $\mathcal{L}_\beta^q [f]=\mathcal{L}_\beta f$, the quotient light ray transform of an equivalence class of a (compactly supported) tensor field~$f$ of degree~$m$. See section~\ref{subsec:lightraytransform}.
\end{itemize}

\subsection{Other operators on tensor fields}

\begin{itemize}
\item $A$, a mixing composed of automorphisms of the tangent bundle. See section~\ref{sec:mixingraydefinition} and equation~\eqref{eq:mixingdefinition}.
\item $A_i$, automorphisms (fiberwise linear bijections) of the tangent bundle. See the beginning of section~\ref{sec:mixingraydefinition}.
\item $\lambda$ and~$\lambda_x$, operators converting $m$-tensor field and $m$-tensor into a function on the tangent bundle and tangent space. See section~\ref{sec:decomp} and equation~\eqref{eq:lambdaoperator}.
\item $\lambda_r=r\circ\lambda$ and $\lambda_{r, x}=r_x\circ\lambda_x$, where $r$ and $r_x$ are the restriction operators on the tangent bundle and tangent space. See section~\ref{sec:decomp}.
\item $A_{k,l}$, the mixing corresponding to the mixed ray transform~$L_{k,l}$. See section~\ref{sec:definitionofmixedray} and equation~\ref{eq:mixedraytwodimensional}.
\item $\sigma$, the usual symmetrization operator of tensor fields. See section~\ref{sec:decomp} and equation~\eqref{eq:sigmadef}.
\item $\widehat{\sigma}_{A, r}$, the projection operator onto~$A^{-1}(\ker(\lambda_r)^\perp)$, related to the mixing ray transform~$I_{A, r}$. See sections~\ref{sec:decomp} and~\ref{sec:mixingraydefinition}, and equations~\eqref{eq:decompositionlambdar} and~\eqref{eq:decompositionlambdarandA}.
\item $\mathcal{H}=\id-\widehat{\sigma}_{A, r}$, an operator projecting $m$-tensor field onto $\ker(\lambda_r\circ A)$. See sections~\ref{sec:mixingraydefinition} and~\ref{sec:solenoidalinjectivity}, and theorem~\ref{thm:linalg}.
\item $\mathcal{D}=A^{-1}\circ\widetilde{A}$, an auxiliary operator related to two admissible mixings~$A$ and~$\widetilde{A}$ of degree~$m$. See section~\ref{sec:mixingraydefinition} and theorem~\ref{thm:linalg}.
\item $\nabla^A=A^{-1}\circ\nabla$, the weighted covariant derivative of a $m$-tensor field where~$A$ is an admissible mixing of degree~$m$. See section~\ref{subsubsection:generalsinjectivity}.
\item $N_{k, l}$, the normal operator of the mixed ray transform~$L_{k, l}$ on compact simple surfaces. See section~\ref{sec:stability} and lemma~\ref{lemma:normaloperator}.
\end{itemize}

\subsection{Other}

\begin{itemize}
\item $\mathcal{F}(X)$, the set of all functions $X\to\C$.
\item $M$ or $(M, g)$, a connected (pseudo-) Riemannian manifold of dimension $n\geq 2$.
\item $SM$, the sphere bundle whose fibers are unit spheres of the tangent spaces. See section~\ref{subsec:geodesicraytransform}.
\item $\X(T_mM)$, the space of all covariant $m$-tensor fields. See section~\ref{sec:notation}.
\item $S_mM$, the space of symmetric $m$-tensor fields. See sections~\ref{sec:notation} and~\ref{sec:decomp}.
\item $C^q(T_mM)$ and $C^q(S_mM)$, the set of $C^q$-smooth (symmetric) $m$-tensor fields where $q\in\mathbb{N}$. See section~\ref{sec:notation}.
\item $H^k(T_mM)$ and $H^k(S_mM)$, the $L^2$-Sobolev space of (symmetric) $m$-tensor field where $k\in\mathbb{N}$. See section~\ref{sec:sobolevnorms}.
\item $P_\eta(T_mM)$ and $P^1_\eta(T_mM)$, the spaces of polynomially decaying $m$-tensor fields on Cartan--Hadamard manifolds. See section~\ref{subsec:geodesicraytransform} and equation~\eqref{eq:cartanhadamardspaces}.
\item $E_\eta(T_mM)$ and $E^1_\eta(T_mM)$, the spaces of exponentially decaying $m$-tensor fields on Cartan--Hadamard manifolds. See section~\ref{subsec:geodesicraytransform} and equation~\eqref{eq:cartanhadamardspaces}.
\item $[f]$ and~$[f]_A$, the equivalence class of the tensor field~$f$, under the relation $f\sim h$ if and only if $f-h\in\ker(\lambda_r\circ A)$. See sections~\ref{sec:decomp}, \ref{sec:mixingraydefinition}, ~\ref{subsubsection:generalsinjectivity} and~\ref{subsec:lightraytransform}.
\end{itemize}

\def\bibfont{\footnotesize}
\bibliographystyle{abbrv}
\bibliography{sample}

\end{document}